\documentclass[11pt,leqno]{amsart}
\usepackage[total={6in,9in},
top=1in, left=0.5in, right=0.5in, bottom=1in]{geometry}
\usepackage{amsmath}
\usepackage[italian, english]{babel}
\usepackage{amsfonts}
\usepackage{amssymb}
\usepackage{dsfont}
\usepackage{mathrsfs}
\usepackage{graphicx}
\usepackage{fancyhdr}
\usepackage{multicol}
\usepackage{enumitem}
\usepackage{amsthm}
\usepackage{empheq}
\usepackage{cases}
\usepackage[all]{xy}
\usepackage{stmaryrd}
\usepackage[colorlinks, citecolor=blue, linkcolor=red]{hyperref}
\usepackage{mathrsfs}
\usepackage{tikz,tikz-cd}
\def\SS{{{\mathbb S}}}
\def\NN{{{\mathbb N}}}

\def\H{{{\mathcal H}}}

\def\K{{\mathcal{K}}}
\def\NK{{\mathcal{NK}}}
\def\AK{{\mathcal{AK}}}

\def\H{{\mathcal{H}}}
\def\JJ{{\mathbf{J}}}

\tikzset{
	subset/.style={
		draw=none,
		edge node={node [sloped, allow upside down, auto=false]{$\subset$}}},
	Subset/.style={
		draw=none,
		every to/.append style={
			edge node={node [sloped, allow upside down, auto=false]{$\subset$}}}
	}
}
\tikzset{
	labl/.style={anchor=south, rotate=90, inner sep=.50mm}
}

\newcommand{\ci}{\mathds{C}}

\newcommand{\del}[1]{\delta_{#1}}

\newcommand{\ricc}{\operatorname{Ric}}
\newcommand{\diver}{\operatorname{div}}

\newcommand{\cp}{\ci\mathbb{P}}

\newcommand{\bigslant}[2]{{\raisebox{.0em}{$#1$}\left/\raisebox{-.0em}{$#2$}\right.}}
\newcommand{\KN}{\mathbin{\bigcirc\mspace{-15mu}\wedge\mspace{3mu}}}

\newcommand{\Ra}{\Rightarrow}
\newcommand{\Lra}{\Leftrightarrow}

\newcommand{\Longra}{\Longrightarrow}

\newcommand{\set}[1]{{\left\{#1\right\}}}               
\newcommand{\pa}[1]{{\left(#1\right)}}                  
\newcommand{\sq}[1]{{\left[#1\right]}}                  
\newcommand{\abs}[1]{{\left|#1\right|}}                 

\newcommand{\eqsist}[1]{\begin{aligned}#1\end{aligned}}  

\newcommand{\ol}[1]{\overline{#1}}
\renewcommand{\hat}[1]{\widehat{#1}}
\renewcommand{\tilde}[1]{\widetilde{#1}}

\newcommand{\qd}[1]{\mathbin{{Q}_{#1}\mspace{-25mu}\tiny\raisebox{0.65ex}{$2$}\mspace{20mu}}}
\newcommand{\qt}[1]{\mathbin{{Q}_{#1}\mspace{-24.7mu}\tiny\raisebox{0.65ex}{$3$}\mspace{20mu}}}
\newcommand{\qq}[1]{\mathbin{{Q}_{#1}\mspace{-25.6mu}\tiny\raisebox{0.7ex}{$4$}\mspace{20mu}}}

\newcommand{\pmat}[1]{{\begin{pmatrix}#1\end{pmatrix}}} 

\newtheorem{theorem}{\textbf{Theorem}}[section]

\newtheorem{proposition}[theorem]{\textbf{Proposition}}
\newtheorem{cor}[theorem]{\textbf{Corollary}}

\theoremstyle{remark}
\newtheorem{rem}[theorem]{\textbf{Remark}}

\numberwithin{equation}{section}
\title[Rigidity results for Riemannian twistor spaces]
{Rigidity results for Riemannian twistor spaces under vanishing
curvature conditions}

\date{\today}


\begin{document}
	\maketitle
	
	\date{\today}
	
	\begin{center}
		\textsc{\textmd{G. Catino\footnote{Politecnico di Milano, Italy.
					Email: giovanni.catino@polimi.it.}, D. Dameno \footnote{Universit\`{a} degli Studi di Milano, Italy.
					Email: davide.dameno@unimi.it.}, P.
				Mastrolia\footnote{Universit\`{a} degli Studi di Milano, Italy.
					Email: paolo.mastrolia@unimi.it.}. }}
	\end{center}

	\begin{abstract}
		In this paper we provide new rigidity results for
		four-dimensional Riemannian manifolds and their twistor spaces.
		In particular, using the moving frame method, we prove that
		$\cp^3$ is the only twistor space whose Bochner tensor is parallel;
		moreover, we classify Hermitian Ricci-parallel
		and locally symmetric twistor spaces and we show
		the nonexistence of conformally flat twistor spaces.
		We also generalize
		a result due to Atiyah, Hitchin and Singer concerning the self-duality
		of a Riemannian four-manifold.
	\end{abstract}

	
	%
	%
	%
	%
	%
	
	\

	\tableofcontents
	
	\
	
	\section{Introduction and main results}\label{secIntro}
	Let $(M,g)$ be an oriented Riemannian manifold of dimension $2n$, with metric
	$g$. The \emph{twistor space} $Z$ associated to $M$ is defined as the set of all the
	couples $(p,J_p)$ such that $p\in M$ and $J_p$ is a
	complex structure on $T_pM$ compatible with $g$, i.e. such that
	$g_p(J_p(X),J_p(Y))=g_p(X,Y)$ for every $X,Y\in T_pM$.
	\footnote{In this work, we call \emph{complex structure} an endomorphism $J_V$ of a
		vector space $V$ such that $J_V^2=-\operatorname{Id}_V$, while we call
		\emph{almost complex structure} a $(1,1)$ tensor field $J$ on a differentiable manifold
		$M$ such that $J$ assigns smoothly, to every point $p$, a complex structure $J_p$
		on $T_pM$.}
	
	Alternatively, we can define $Z$ in an equivalent way
	as
	\[
	Z=\bigslant{O(M)}{U(n)},
	\]
	where $O(M)$ denotes the orthonormal frame bundle over $M$ and
	the unitary group $U(n)$ is identified with a subgroup of $SO(2n)$
	(see \cite{debnan} for further details).
	
	These structures, introduced
	by Penrose (\cite{penrose}) as an attempt to define an innovative
	framework for Physics, have been the subject of many investigations
	by the mathematical community, in virtue of the numerous geometrical and algebraic tools involved in the definition of their properties. In 1978, Atiyah, Hitchin and
	Singer (\cite{athisin}) adapted Penrose's twistor theory to the Riemannian context,
	introducing the concept of twistor space associated to a Riemannian four-manifold
	and paving the way for many researches about this subject.

	The orientation on $M$ implies that $O(M)$ has two connected components,
	$O(M)_+$ and $O(M)_-$, and
	therefore we can define the two connected components of $Z$
	\[
	Z_{\pm}=\bigslant{O(M)_{\pm}}{U(n)}=\bigslant{SO(M)}{U(n)},
	\]
	where $SO(M)$ is the orthonormal oriented frame bundle over $M$.
	We choose the component $Z_-$ to be the twistor space of $(M,g)$
	(see also \cite{cdm} and \cite{salamon}). It is possible to define
	a natural family of Riemannian metrics $g_t$ on $Z_-$,
	where $t>0$ (\cite{debnan} \cite{jenrig}); from now on, we
	systematically use the notation $(Z,g_t)$ to denote the twistor space
	$Z_-$ endowed with the Riemannian metric $g_t$.
	
	
	
	In general, if $(M,g)$ is a Riemannian manifold of dimension
	$m\geq 3$, the Riemann curvature tensor $\operatorname{Riem}$ on $M$
	admits the well known decomposition
	\[
	\operatorname{Riem}=\operatorname{W}+\frac{1}{m-2}\operatorname{Ric}\KN g
	-\frac{S}{2(m-1)(m-2)}g\KN g,
	\]
	where $\operatorname{W}$, $\operatorname{Ric}$ and $S$ denote the
	\emph{Weyl tensor}, the \emph{Ricci tensor} and the \emph{scalar curvature}
	of $M$, respectively, and $\KN$ is the \emph{Kulkarni-Nomizu product}. Moreover,
	the Riemann curvature tensor defines a symmetric linear operator
	from the bundle of two-forms $\Lambda^2$ to itself
	\begin{align*}
		\mathcal{R}\colon\Lambda^2&\longrightarrow\Lambda^2\\
		\gamma&\longmapsto\mathcal{R}(\gamma)=\dfrac{1}{4}R_{ijkt}\gamma_{kt}\theta^i\wedge\theta^j,
	\end{align*}
	where $\{\theta^i\}_{i=1,\ldots,m}$ is a local orthonormal coframe on an
	open set $U\subset M$, with dual frame $\{e_i\}_{i=1,\ldots,m}$,
	$\gamma_{kt}=\gamma(e_k,e_t)$ and $R_{ijkt}$ are the components
	of the Riemann tensor with respect to the coframe $\{\theta^i\}$.
	
	If $m=4$ and $M$ is oriented, $\Lambda^2$ splits, {\itshape via}
	the Hodge $\star$ operator, into the direct sum of
	two subbundles $\Lambda_+$ and $\Lambda_-$
	.
	This implies that the Riemann curvature operator
	$\mathcal{R}$ assumes a block matrix form
	\[
	\mathcal{R}=
	\left(
	\begin{array}{cc}
		A & B^T\\
		B & C
	\end{array}
	\right),
	\]
	where $A$ (resp, $C$) is a
	symmetric endomorphism of $\Lambda_+$ (resp., $\Lambda_-$) and $B$ is
	a symmetric linear map from $\Lambda_+$ to $\Lambda_-$ (see \cite{athisin},
	\cite{besse} and \cite{singthor}).
	Moreover, $\operatorname{tr}A=\operatorname{tr}C=\frac{S}{4}$.
	This corresponds to a decomposition of the Weyl tensor into a sum
	\[
	W=W^++W^-,
	\]
	where $W^+$ (resp., $W^-$) is called the {\em self-dual} (resp., {\em anti-self-dual}) {\em part}
	of $W$. If $W^+=0$ (resp., $W^-=0$), we say that $M$ is an {\em anti-self-dual}
	(resp., {\em self-dual}) {\em manifold}. If we consider the symmetric linear operators induced
	by $W^+$ and $W^-$, we have that their representative matrices are
	$A-\frac{S}{12}I_3$ and $C-\frac{S}{12}I_3$, respectively, with respect to any
	positively oriented local orthonormal coframe; thus, $(M,g)$ is self-dual (resp.,
	anti-self-dual) if and only if $C=\frac{S}{12}I_3$ (resp., $A=\frac{S}{12}I_3$).
	Note that, if the coframe is negatively oriented, $A$ and $C$ need to be
	exchanged in the previous statements.

	In this paper, starting from our previous work \cite{cdm}, we focus
	our attention on some rigidity results concerning twistor spaces
	satisfying vanishing conditions on relevant geometric tensors, such
	as the Weyl tensor, the Bochner tensor and the covariant derivatives
	of the Ricci tensor and the Riemann tensor. For instance, we are
	able to show the following results:
	\begin{itemize}
		 \item nonexistence of locally conformally flat twistor spaces;
		 \item a twistor space is Bochner-flat if and only if the underlying manifold is homotetically isometric to $\SS^4$;
		 \item characterization of Ricci-parallel and locally symmetric twistor spaces;
		 \item a generalization of Atiyah-Hitchin-Singer result, using the divergences of the Nijenhuis tensor(s).
	\end{itemize}
	 The paper is organized as follows: in Section 2, we show that, given
	 a Riemannian four-manifold $(M,g)$, its twistor space $(Z,g_t)$ cannot be locally conformally flat for any $t>0$.
	
	 Section 3 is devoted to the characterization of Bochner-flat
	 twistor spaces: in particular, we show that the only Bochner-parallel
	 twistor space is "essentially" $\cp^3$, which is the one associated
	 to the four-sphere $\mathds{S}^4$.
	
	 In Section 4 we consider Ricci parallel and locally symmetric twistor
	 spaces, providing rigidity results for twistor spaces whose
	 Atiyah-Hitchin-Singer almost complex structure $J^+=J$ is integrable
	 (see Appendix \ref{appb} for details).
	
	 In Section 5 we prove a general quadratic formula for
	 $\abs{\nabla J}^2$; moreover, we generalize the necessary and sufficient
	 condition for the integrability of $J$, first proven by Atiyah,
	 Hitchin and Singer \cite{athisin}, through a vanishing condition
	 on the divergences of the associated Nijenhuis tensor.
	 We also prove a new result concerning the Nijenhuis tensor of the
	 Eells-Salamon almost complex structure $J^-=\JJ$ (see \cite{elsal}).
	
	 To keep the paper self-contained as much as possible,
	 we provide also five brief appendices
	 devoted to technicalities and some heavy computations (for instance,
	 the list of the local components of the Weyl tensor of a twistor
	 space $(Z,g_t)$).
	
	\

\

\

\

\section{Locally conformally flat twistor spaces}

In this section, we want to show that the twistor space $(Z,g_t)$ associated
to a Riemannian four-manifold $(M,g)$ is never locally conformally flat
for any $t>0$.
By Weyl-Schouten
Theorem, we know that a Riemannian manifold of dimension $n\geq 4$ is
locally conformally flat if and only if its Weyl tensor $\operatorname{W}$
vanishes identically (for a proof, see \cite{jeromin} or \cite{lee2}).

Before we state the main result of this section, let us
recall the transformation laws
for the matrices $A$ and $B$ appearing in the decomposition of
the Riemann curvature operator: we know that, given a
local orthonormal frame $e\in O(M)_-$,
if we choose another frame $\tilde{e}\in O(M)_-$, the
change of frames is determined by a matrix $a\in SO(4)$ and that
the matrices $A$ and $B$ transform
according to the equations
\begin{equation} \label{transfabc}
	\tilde{A}=a_+^{-1}Aa_+, \qquad \tilde{B}=a_-^{-1}Ba_+
\end{equation}
where $SO(3)\times SO(3)$\reflectbox{$\in$}$(a_+,a_-)=\mu(a)$ and
$\mu$ is a surjective homomorphism from $SO(4)$ to $SO(3)\times SO(3)$
induced by the universal covers of $SO(4)$ and $SO(3)$ (see
\cite{besse}, \cite{cdm} and \cite{salamon} for a detailed description).

For the sake of simplicity, throughout the paper we adopt the following
notation
\begin{align} \label{riemq}
	\qd{ab}&:=R_{12ab}+R_{34ab};\\
	\qt{ab}&:=R_{13ab}+R_{42ab}; \notag\\
	\qq{ab}&:=R_{14ab}+R_{23ab}. \notag
\end{align}
We also compute the differentials of the components listed in \eqref{riemq}:
\begin{align} \label{covderq}
	d\qd{ab} &= {\qd{ab}}_{,c}\omega^c + \qd{ac}\omega^c_b + \qd{cb}\omega^c_a+
	\qq{ab}({\omega^1_3+\omega^4_2}) - \qt{ab}({\omega^1_4+\omega^2_3}); \\
	d\qt{ab} &= {\qt{ab}}_{,c}\omega^c + \qt{ac}\omega^c_b + \qt{cb}\omega^c_a-
	\qq{ab}({\omega^1_2+\omega^3_4}) + \qd{ab}({\omega^1_4+\omega^2_3}); \notag\\
	d\qq{ab} &= {\qq{ab}}_{,c}\omega^c + \qq{ac}\omega^c_b + \qq{cb}\omega^c_a + \qt{ab}({\omega^1_2+\omega^3_4}) - \qd{ab}({\omega^1_3+\omega^4_2}). \notag,
\end{align}
where $\{\omega^1,...,\omega^4\}$ is a local orthonormal coframe and 
$\omega_j^i$ are the associated Levi-Civita connection 1-forms.

%
%
%
Now, we can state the following result, which is new, to the best of
our knowledge:
\begin{theorem} \label{locconftwist}
	Let $(M,g)$ be a Riemannian four-manifold and $(Z,g_t)$ be its twistor
	space. Then, $(Z,g_t)$ is not locally conformally flat for any $t>0$.
\end{theorem}
\begin{proof}
	Let us suppose that $(Z,g_t)$ is locally conformally flat, i.e., by
	Weyl-Schouten Theorem, $\ol{\operatorname{W}}\equiv 0$ on $Z$. By the
	vanishing of the coefficients $\ol{W}_{ab56}$ in \eqref{weyltwist},
	we obtain the system
	\[
	\begin{cases}
		\qd{12}=\frac{t^2}{4}(\qt{1c}\qq{2c}-\qt{2c}\qq{1c})\\
		\qd{34}=\frac{t^2}{4}(\qt{3c}\qq{4c}-\qt{4c}\qq{3c})
	\end{cases};
	\]
	expliciting the right-hand sides and then summing the equations,
	we derive the equality
	\[
	2A_{11}+t^2(A_{23}^2-A_{22}A_{33})=0.
	\]
	Note that $\ol{\operatorname{W}}\equiv 0$ is a global condition: in
	particular, this means that the equation above must hold for every
	$p\in M$ (it suffices to consider the pullback maps \emph{via}
	any section of the twistor bundle). Moreover, since the locally
	conformally flatness is a frame-independent condition, the equation
	holds for every local negatively oriented orthonormal frame
	$e\in O(M)_-$. In particular, since $A$ is a symmetric matrix, we
	have that the equality holds for every frame $e$ such that $A$ is
	diagonal; in this situation, we have that
	\[
	2A_{11}-t^2A_{22}A_{33}=0,
	\]
	for every frame with respect to which $A$ is diagonal.
	By \eqref{transfabc}, we can exchange the diagonal entries of $A$
	with suitable changes of frames in order to obtain the additional
	equations
	\begin{align*}
		0&=2\tilde{A}_{11}-t^2\tilde{A}_{22}\tilde{A}_{33}=
		2A_{22}-t^2A_{11}A_{33}\\
		0&=2\hat{A}_{11}-t^2\hat{A}_{22}
		\hat{A}_{33}=
		2A_{33}-t^2A_{11}A_{22},
	\end{align*}
	where $\tilde{A}_{ij}$ and $\hat{A}_{ij}$ are the entries of the
	matrix $A$ with respect to some frames $\tilde{e}$ and $\hat{e}$,
	respectively. At a point $p\in M$, since $t>0$,
	the system of these three equations
	admits three distinct solutions:
	\begin{enumerate}
		\item $A_{11}=A_{22}=A_{33}=0$;
		\item $A_{11}=A_{22}=A_{33}=\dfrac{2}{t^2}$;
		\item two diagonal entries out of three are equal to
		$-\dfrac{2}{t^2}$, while the third is equal to $\dfrac{2}{t^2}$.
	\end{enumerate}
	This means that, at $p\in M$, the scalar curvature $S$ of $(M,g)$
	can attain the values $0$, $24/t^2$ or $-8/t^2$. Since the
	scalar curvature is a smooth function on $M$ and,
	for every point of $M$, one of the three equations must hold,
	we can conclude that $S$ is constant on
	$M$: indeed, the possible values for $S$ are finitely many, therefore,
	if $S(p)\neq S(p')$ for $p, p'\in M$, $S$ would not be a smooth function.
	
	First, let us prove that the first two cases lead to a contradiction.
	Note that, in this situation, $A$ is a scalar matrix for every point
	$p\in M$ (and, by \eqref{transfabc}, for every frame), which means
	that $(M,g)$ is a self-dual manifold. By the vanishing of
	the components $\ol{W}_{5ab5}$ and $\ol{W}_{6ab6}$, if $a\neq b$,
	we obtain
	\[
	0=\ol{W}_{5ab5}+\ol{W}_{6ab6}=
	\dfrac{1}{2}R_{ab}-\dfrac{t^2}{2}(\qt{ac}\qt{bc}+\qq{ac}\qq{bc});
	\]
	in particular, for $(a,b)=(1,2)$ and $(a,b)=(3,4)$, by
	the self-duality condition we can compute
	\begin{align*}
		\qq{13}-\qt{14}&=R_{12}=\dfrac{t^2 S}{6}(\qq{13}-\qt{14})\\
		\qq{13}+\qt{14}&=R_{34}=\dfrac{t^2 S}{6}(\qq{13}+\qt{14}),
	\end{align*}
	which imply immediately $\qt{14}=\qt{23}=\qq{13}=\qq{42}=0$ on $M$.
	This is equivalent to say that the entries $B_{23}$ and $B_{32}$
	of the matrix $B$ vanish
	identically on $M$; since this is a global condition, by suitable change
	of frames, equation \eqref{transfabc} implies that the matrix $B$ is
	the zero matrix, i.e. $(M,g)$ is an Einstein manifold (see also
	\cite{cdm} for a detailed proof). However, we have
	that
	\[
	0=\ol{W}_{5656}\Longra\abs{\qt{ab}}^2+\abs{\qq{ab}}^2=\dfrac{2S}{3t^2}+
	\dfrac{8}{t^4};
	\]
	the left-hand side of the second equation is equal to $(S^2/18)$ for
	an Einstein, self-dual manifold, hence, for $S$ equal to $0$ or to
	$24/t^2$, we get a contradiction.
	
	Thus, we can choose a frame $e$ with respect to which $A$ is diagonal
	and
	\[
	A_{11}=-A_{22}=A_{33}=-\dfrac{2}{t^2}.
	\]
	If we consider again the equations $\ol{W}_{5ab5}=\ol{W}_{6ab6}=0$,
	we obtain
	\begin{align*}
		\dfrac{1}{2}(\qq{13}-\qt{14})&=\dfrac{1}{2}R_{12}=
		\dfrac{t^2}{2}(\qt{1c}\qt{2c}+\qq{1c}\qq{2c})=2(\qq{13}+\qt{14})\\
		\dfrac{1}{2}(\qq{13}+\qt{14})&=\dfrac{1}{2}R_{34}=
		\dfrac{t^2}{2}(\qt{3c}\qt{4c}+\qq{3c}\qq{4c})=2(\qq{13}-\qt{14}),
	\end{align*}
	which obviously imply $\qt{14}=\qq{13}=0$, i.e. $B_{32}=B_{23}=0$
	in the chosen frame $e$. In fact, we can say more:
	the equalities $B_{32}=B_{23}=0$ hold for every frame $e'$ with respect
	to which the matrix $A$ is in diagonal form
	with $A_{11}=-A_{22}=A_{33}=-2/t^2$.
	Note that we can choose suitable change of frames such that
	$\tilde{A}=A$, where $\tilde{A}$ is the matrix associated to
	the transformed frame $\tilde{e}$: indeed, it suffices to
	choose $a_+=I_3$ in \eqref{transfabc}.
	
	Therefore, with suitable choices of $a_-$ and putting
	$a_+=I_3$ in \eqref{transfabc}, it is immediate to show that
	\[
	B_{12}=B_{13}=B_{22}=B_{23}=B_{32}=B_{33}=0
	\]
	for a frame $e$ with respect to which $A$ is in diagonal form
	with $A_{11}=-A_{22}=A_{33}=-2/t^2$.
	
	Finally, let us compute
	\begin{align*}
	0=\ol{W}_{5656}\Longra \dfrac{128}{t^4}=
	\abs{\qt{ab}}^2+\abs{\qq{ab}}^2=\dfrac{2S}{3t^2}+\dfrac{8}{t^4}=
	\dfrac{8}{3t^4},
	\end{align*}
	which is obviously impossible. Thus, $(Z,g_t)$ cannot be locally
	conformally flat.
\end{proof}
By well-known results due to G\l odek (see \cite{glodek}), 
Derdzi\'{n}ski and Roter (see \cite{derdroter} and \cite{roter}), 
it is immediate to show the following
\begin{cor} \label{weylpartwist}
	A twistor space $(Z,g_t)$ is conformally symmetric, i.e.
	$\nabla\ol{\operatorname{W}}\equiv 0$,
	if and only if it is locally symmetric,
	i.e. $\nabla\ol{\operatorname{Riem}}\equiv 0$.
\end{cor}

\section{Bochner-flat twistor spaces} \label{bochner}

Let $(N,g,J)$ be a almost Hermitian manifold of dimension $2n$. We can define
the \emph{Bochner tensor} $\operatorname{B}$ of $N$ as the $(0,4)$-tensor
whose components with respect to a local orthonormal frame are

\begin{align} \label{bochneralmherm}
B_{pqrs}=R_{pqrs}&+\dfrac{1}{2(n+2)}\left[\del{ps}R_{qr}-\del{pr}R_{qs}
+\del{qr}R_{ps}-\del{qs}R_{pr}+\right.\\
&\left. +J_s^pJ_r^tR_{qt}-J_r^pJ_s^tR_{qt}-2J_q^pJ_s^tR_{rt}+\right. \notag
\\
&\left. +J_r^qJ_s^tR_{pt}-J_s^qJ_r^tR_{pt}-2J_s^rJ_q^tR_{pt}\right]+ \notag
\\
&-\dfrac{S}{4(n+1)(n+2)}\sq{\del{ps}\del{qr}-\del{pr}\del{qs}+J_s^pJ_r^q
	-J_r^pJ_s^q-2J_q^pJ_s^r}. \notag
\end{align}

This tensor was first introduced by Bochner as a "complex analogue" of the Weyl
tensor \cite{boch}. It is important to note that some authors define the
Bochner tensor as $-\operatorname{B}$,
because of a different convention for the sign of
the Riemann tensor (see, for instance, \cite{tachi} and \cite{tanno}).

We say that $N$ is a \emph{Bochner-flat manifold} if $\operatorname{B}$
vanishes identically,
i.e. if $B_{pqrs}=0$ for every $1\leq p,q,r,s\leq 2n$. It is known that, in
general, the Bochner tensor does not satisfy the same symmetries as the
Riemann tensor (see, for instance, \cite{vanyano}). However, if $N$ is
Bochner-flat, by \eqref{bochneralmherm} we obtain
\begin{align} \label{symmboch}
R_{pqrs}=&-\dfrac{1}{2(n+2)}\left[\del{ps}R_{qr}-\del{pr}R_{qs}
+\del{qr}R_{ps}-\del{qs}R_{pr}+\right.\\
&\left. +J_s^pJ_r^tR_{qt}-J_r^pJ_s^tR_{qt}-2J_q^pJ_s^tR_{rt}+\right. \notag
\\
&\left. +J_r^qJ_s^tR_{pt}-J_s^qJ_r^tR_{pt}-2J_s^rJ_q^tR_{pt}\right]+ \notag
\\
&+\dfrac{S}{4(n+1)(n+2)}\sq{\del{ps}\del{qr}-\del{pr}\del{qs}+J_s^pJ_r^q
	-J_r^pJ_s^q-2J_q^pJ_s^r}, \notag
\end{align}
which means that the right-hand side of \eqref{symmboch} satisfies the
same symmetries as the Riemann tensor.

Now, let $(M,g)$ be a four-dimensional Riemannian manifold and let
$(Z,g_t,J)$ be its twistor space, regarded as an almost Hermitian manifold.
It is known that $(Z,g_t,J)$ is a K\"{a}hler-Einstein manifold if and only if
$(M,g)$ is an Einstein, self-dual manifold with scalar curvature $S=12/t^2$
(see, for instance, \cite{cdm}, \cite{davmus1}, \cite{mus} and
Corollary \ref{kahlertwist} of this paper). Let us
suppose that $(Z,g_t,J)$ is a K\"{a}hler-Einstein manifold and let
$\ol{\operatorname{B}}$ be its Bochner tensor.
Under these hypotheses, we can compute
the components $\ol{B}_{pqrs}$:
\begin{align} \label{bochkahltwist}
	\ol{B}_{pqrs}&=0, \mbox{ if at least one of the indices
		is equal to 5 or 6};\\
	\ol{B}_{abcd}&=R_{abcd}-\dfrac{1}{t^2}(\delta_{ac}\delta_{bd}-
	\delta_{ad}\delta_{bc})\notag
\end{align}
(note that, in this case, the Bochner tensor satisfies the same
symmetries as the Riemann tensor. See also Remark \ref{bochnsymm}).
By direct inspection of these components and by recalling that
$\mathds{S}^4$ is the only four-dimensional space form with positive sectional
curvature, up to isometries, one can show the following
\begin{proposition} \label{bochflatkahl}
Let $(M,g)$ be a Riemannian four-manifold
such that its twistor space $(Z,g_t,J)$ is K\"{a}hler-Einstein.
Then $(Z,g_t,J)$ is Bochner-flat if and only if
$(M,g)$ is isometric to $\mathds{S}^4$, with its canonical Riemannian
metric.
\end{proposition}

It is natural to ask whether Proposition \ref{bochflatkahl} can be
generalized or not if there are no hypothesis on the almost complex
structure $J$: more precisely, our goal is to characterize almost
Hermitian, Bochner-flat twistor spaces. Rather surprisingly,
it turns out that $\SS^4$ is the only Riemannian four-manifold whose
twistor space is Bochner-flat.

First, let us define
the covariant derivative $\nabla{\operatorname{B}}$ of the Bochner tensor
$\operatorname{B}$ of an almost Hermitian manifold $(N,g,J)$,
whose components with respect to a local orthonormal coframe are
\begin{align} \label{nablaboch}
B_{pqrs,u}=R_{pqrs,u}&+\dfrac{1}{2(n+2)}\left[\del{ps}R_{qr,u}-
\del{pr}R_{qs,u}+\del{qr}R_{ps,u}-\del{qs}R_{pr,u}+ \right.\\
&\left.+R_{qt}(J_r^tJ_{s,u}^p+J_s^pJ_{r,u}^t)+J_s^pJ_r^tR_{qt,u}+\right.
\notag \\
&\left.-R_{qt}(J_s^tJ_{r,u}^p+J_r^pJ_{s,u}^t)-J_r^pJ_s^tR_{qt,u}+\right.
\notag\\
&\left.-2R_{rt}(J_s^tJ_{q,u}^p+J_q^pJ_{s,u}^t)-2J_q^pJ_s^tR_{rt,u}+\right.
\notag\\
&\left.+R_{pt}(J_s^tJ_{r,u}^q+J_r^qJ_{s,u}^t)+J_r^qJ_s^tR_{pt,u}+\right.
\notag\\
&\left.-R_{pt}(J_r^tJ_{s,u}^q+J_s^qJ_{r,u}^t)-J_s^qJ_r^tR_{pt,u}+\right.
\notag\\
&\left.-2R_{pt}(J_q^tJ_{s,u}^r+J_s^rJ_{q,u}^t)-2J_s^rJ_q^tR_{pt,u}\right]+
\notag\\
&-\dfrac{S_u}{4(n+1)(n+2)}\sq{\del{ps}\del{qr}-\del{pr}\del{qs}+
	J_s^pJ_r^q-J_r^pJ_s^q-2J_q^pJ_s^r}\notag\\
&-\dfrac{S}{4(n+1)(n+2)}[J_r^qJ_{s,u}^p+J_s^pJ_{r,u}^q-
J_s^qJ_{r,u}^p-J_r^pJ_{s,u}^q-2J_s^rJ_{q,u}^p-2J_q^pJ_{s,u}^r], \notag
\end{align}
and
\begin{align*}
\nabla\ricc&=R_{pq,t}\theta^t\otimes\theta^p\otimes\theta^q\\
\nabla J&=J_{q,t}^p\theta^t\otimes\theta^q\otimes e_p\\
dS&=S_u\theta^u.
\end{align*}
We say that $(N,g,J)$ is \emph{Bochner-parallel} if
$\nabla\operatorname{B}\equiv 0$; in this case,
by \eqref{nablaboch} the components
of $\nabla\operatorname{Riem}$ satisfy the equation
\begin{align} \label{nablabochsymm}
R_{pqrs,u}=-\dfrac{1}{2(n+2)}&\left[\del{ps}R_{qr,u}-
\del{pr}R_{qs,u}+\del{qr}R_{ps,u}-\del{qs}R_{pr,u}+ \right.\\
&\left.+R_{qt}(J_r^tJ_{s,u}^p+J_s^pJ_{r,u}^t)+J_s^pJ_r^tR_{qt,u}+\right.
\notag \\
&\left.-R_{qt}(J_s^tJ_{r,u}^p+J_r^pJ_{s,u}^t)-J_r^pJ_s^tR_{qt,u}+\right.
\notag\\
&\left.-2R_{rt}(J_s^tJ_{q,u}^p+J_q^pJ_{s,u}^t)-2J_q^pJ_s^tR_{rt,u}+\right.
\notag\\
&\left.+R_{pt}(J_s^tJ_{r,u}^q+J_r^qJ_{s,u}^t)+J_r^qJ_s^tR_{pt,u}+\right.
\notag\\
&\left.-R_{pt}(J_r^tJ_{s,u}^q+J_s^qJ_{r,u}^t)-J_s^qJ_r^tR_{pt,u}+\right.
\notag\\
&\left.-2R_{pt}(J_q^tJ_{s,u}^r+J_s^rJ_{q,u}^t)-2J_s^rJ_q^tR_{pt,u}\right]+
\notag\\
&+\dfrac{S_u}{4(n+1)(n+2)}\sq{\del{ps}\del{qr}-\del{pr}\del{qs}+
	J_s^pJ_r^q-J_r^pJ_s^q-2J_q^pJ_s^r}\notag\\
&+\dfrac{S}{4(n+1)(n+2)}[J_r^qJ_{s,u}^p+J_s^pJ_{r,u}^q-
J_s^qJ_{r,u}^p-J_r^pJ_{s,u}^q-2J_s^rJ_{q,u}^p-2J_q^pJ_{s,u}^r]. \notag
\end{align}
Before we state and prove the main result of this section, we need
the following
\begin{theorem} \label{complexlinricci}
	Let $(M,g)$ be a Riemannian four-manifold and $(Z,g_t,J)$ be its
	twistor space. Then, the Ricci tensor $\ol{\ricc}$ of $Z$ is
	complex linear, i.e.
	\begin{equation} \label{riccicommj}
	\ol{R}_{pt}J_q^t+\ol{R}_{qt}J_p^t=0 \mbox{ on } Z \mbox{ for every }
	p,q=1,...,6,
	\end{equation}
	if and only if $(M,g)$ is an Einstein,
	self-dual manifold.
\end{theorem}
\begin{proof}
	If $(M,g)$ is an Einstein, self-dual manifold, the validity of \eqref{riccicommj}
	can be immediately shown by a direct inspection of the components
	listed in \eqref{ricctwist}.
	
	Thus, let us suppose that \eqref{riccicommj} holds on $Z$. First,
	note that
	\begin{align*}
		0&=\ol{R}_{1t}J_1^t+\ol{R}_{1t}J_1^t=2\ol{R}_{1t}J_1^t=\ol{R}_{12}\\
		0&=\ol{R}_{3t}J_3^t+\ol{R}_{3t}J_3^t=2\ol{R}_{3t}J_3^t=\ol{R}_{34}
	\end{align*}
on $Z$; by \eqref{ricctwist}, we can write
\begin{align*}
	R_{12}&=\dfrac{t^2}{2}(\qt{1c}\qt{2c}+\qq{1c}\qq{2c})\\
	R_{34}&=\dfrac{t^2}{2}(\qt{3c}\qt{4c}+\qq{3c}\qq{4c}).
\end{align*}
Subtracting the first equation from the second, we obtain
\[
B_{32}=t^2(B_{22}B_{32}+B_{23}B_{33});
\]
note that this is a global condition on the entries of the matrix $B$,
which means that it holds for every choice of local orthonormal frame.
By \eqref{transfabc}, we can choose
\[
a_-=\pmat{
1 & 0 & 0\\
0 & 0 & 1\\
0 & -1 & 0
},
\qquad
a_+=I_3,
\]
where $I_3$ is the $3\times3$ identity matrix, to compute
\[
\tilde{B}=\pmat{
B_{11} & B_{12} & B_{13}\\
-B_{31} & -B_{32} & -B_{33}\\
B_{21} & B_{22} & B_{23}
}
\]
and to obtain
\[
B_{22}=\tilde{B}_{32}=t^2(\tilde{B}_{22}\tilde{B}_{32}+
\tilde{B}_{23}\tilde{B}_{33})=
-t^2(B_{32}B_{22}+B_{33}B_{23})=-B_{32}.
\]
Now, choosing
\[
a_-=\pmat{
	-1 & 0 & 0\\
	0 & -1 & 0\\
	0 & 0 & 1
},
\qquad
a_+=I_3,
\]
we get
\[
\tilde{B}=\pmat{
	-B_{11} & -B_{12} & -B_{13}\\
	-B_{21} & -B_{22} & -B_{23}\\
	B_{31} & B_{32} & B_{33}
}
\]
and
\[
B_{32}=\tilde{B}_{32}=t^2(\tilde{B}_{22}\tilde{B}_{32}+
\tilde{B}_{23}\tilde{B}_{33})=
-t^2(B_{32}B_{22}+B_{33}B_{23})=-B_{32};
\]
therefore, we conclude that $B_{22}=B_{32}=0$.

Now, we can choose the change of frames determined by
\[
a_-=I_3, \quad
a_+=\pmat{
1 & 0 & 0\\
0 & 0 & -1\\
0 & 1 & 0
}\quad \mbox{ and } \quad
a_-=\pmat{
1 & 0 & 0\\
0 & 0 & 1\\
0 & -1 & 0
}, \quad
a_+=\pmat{
	1 & 0 & 0\\
	0 & 0 & -1\\
	0 & 1 & 0
}
\]
to compute
\begin{align*}
	B_{33}=t^2B_{23}B_{33}=-B_{23},
\end{align*}
which obviously implies that $B_{23}=B_{33}=0$. By an analogous computations,
we can also obtain $B_{21}=B_{31}=0$. Finally, choosing
\[
a_-=\pmat{
0 & 0 & 1\\
0 & 1 & 0\\
-1 & 0 & 0
}, \qquad
a_+=I_3
\]
and repeating the argument above, we conclude that $B_{11}=B_{12}=B_{13}=0$.
Thus, for every $p\in M$ there exists a local frame such that $B=0$, i.e.
$(M,g)$ is an Einstein manifold.

Now, by \eqref{riccicommj}, we obtain
\[
0=\ol{R}_{5t}J_6^t+\ol{R}_{6t}J_5^t\Longra \ol{R}_{55}=\ol{R}_{66}
\Longra\abs{\qt{ab}}^2=\abs{\qq{ab}}^2;
\]
since $(M,g)$ is an Einstein manifold, this equation can be rewritten as
\[
(A_{12})^2+(A_{22})^2=(A_{13})^2+(A_{33})^2,
\]
which is another global condition and, therefore, does not depend
on the choice of the local frame. In particular, for every frame with
respect to which $A$ is diagonal, we have that $(A_{22})^2=(A_{33})^2$ and,
since we can exchange the diagonal entries of $A$, we can conclude
\[
(A_{11})^2=(A_{22})^2=(A_{33})^2.
\]
If $A_{11}=A_{22}=A_{33}$, then this holds for every local orthonormal frame
by \eqref{transfabc}, hence $(M,g)$ is self-dual and the claim is proven.

Thus, without loss of generality, we may suppose that,
for instance, $A_{22}\neq A_{33}$ at a point
$p\in M$ for some local frame $e\in O(M)_-$ with respect to which
$A$ is diagonal. By the equation above,
it is obvious that the diagonal entries satisfy
\[
A_{11}=A_{22}=-A_{33}=x\neq 0;
\]
if we choose the change of frames determined by
\[
a_+=\pmat{
0 & -1 & 0\\
\frac{1}{\sqrt{2}} & 0 & -\frac{1}{\sqrt{2}}\\
\frac{1}{\sqrt{2}} & 0 & \frac{1}{\sqrt{2}}
},
\]
we obtain
\[
\tilde{A}=\pmat{
0 & 0 & -\frac{x}{2}\\
0 & x & 0\\
-\frac{x}{2} & 0 & 0
}
\]
and
\[
(\tilde{A}_{12})^2+(\tilde{A}_{22})^2=(\tilde{A}_{13})^2+(\tilde{A}_{33})^2
\Longra x^2=\dfrac{x^2}{4},
\]
which is true if and only if $x=0$ and leads to a contradiction.
Therefore, diagonalizing $A$ we obtain a scalar matrix on $M$, which is
equivalent to say that $(M,g)$ is a self-dual manifold.
\end{proof}
We are now ready to state the following
\begin{theorem} \label{nablabochflat}
Let $(M,g)$ be a Riemannian four-manifold and $(Z,g_t,J)$ be its
twistor space. Then, $(Z,g_t,J)$ is Bochner-parallel if and only
if $(M,g)$ is homotetically isometric to $\mathds{S}^4$
with its canonical metric. In particular, the only Bochner-parallel
twistor space is $\cp^3$ endowed with the Fubini-Study metric.
\end{theorem}
An immediate consequence of Theorem \ref{nablabochflat} is
the following
\begin{cor} \label{bochflat}
$(Z,g_t,J)$ is Bochner-flat if and only if $(M,g)$ is homotetically isometric
to $\mathds{S}^4$ with its canonical metric.
\end{cor}
\begin{proof}[Proof of Theorem \ref{nablabochflat}]
First, note that, in order to prove the claim, it is sufficient
to show that, if $(Z,g_t,J)$ is Bochner-parallel, then it is
a K\"{a}hler-Einstein manifold, i.e. $(M,g)$ is an Einstein,
self-dual manifold with scalar curvature $S$ equal to
$12/t^2$: indeed,
by \eqref{bochkahltwist}, it is immediate to show that, if
$(Z,g_t,J)$ is a K\"{a}hler-Einstein manifold, then it is Bochner-parallel
if and only if it is Bochner-flat (it is sufficient to check the
components $\ol{B}_{pqr5,t}$); therefore, this ends the proof by
Proposition \ref{bochflatkahl}.

Let us consider the local expression of
$\nabla\operatorname{B}$ in \eqref{nablabochsymm}:
the right-hand side satisfies the same symmetries as the
components of $\nabla\operatorname{Riem}$. Thus,
recalling that the Riemann tensor is skew-symmetric with
respect to the last two indices, we obtain
\[
0=\ol{R}_{pqrs,u}+\ol{R}_{qprs,u}=
\dfrac{1}{5}[J_{s,u}^r(\ol{R}_{pt}J_q^t+\ol{R}_{qt}J_p^t)+
J_s^r(\ol{R}_{pt}J_{q,u}^t+\ol{R}_{qt}J_{p,u}^t+
J_q^t\ol{R}_{pt,u}+J_p^t\ol{R}_{qt,u})]
\]
for every $p,q=1,...,6$.

If we consider a pair of indices $(r,s)$ such that
$J_s^r=0$, the equation becomes
\[
J_{s,u}^r(\ol{R}_{pt}J_q^t+\ol{R}_{qt}J_p^t)=0.
\]
If $J_{s,u}^r\neq 0$ for some $r,s,u$ and for every local frame $e$, we
conclude that $\ol{R}_{pt}J_q^t+\ol{R}_{qt}J_p^t=0$ for
every $p,q$.
On the other hand, if $e'$ is a frame with respect to which $J_{s,u}^r=0$
for every $r,s,u$, we obtain
\begin{align*}
	\qt{13}&=\qt{42}=\qq{14}=\qq{23}=\dfrac{1}{t^2}\\
	\qt{12}&=\qt{34}=\qq{12}=\qq{34}=\qt{14}=
	\qt{23}=\qq{13}=\qq{42}=0,
\end{align*}
which obviously imply that
\[
\ol{R}_{12}=\ol{R}_{34}=0 \mbox { and } \ol{R}_{55}=\ol{R}_{66}
\]
for the chosen frame.
Thus, these equations hold on $Z$ for every choice of
$e\in O(M)$: therefore, we can repeat the argument exploited in the proof of
Theorem \ref{complexlinricci} to conclude that
$(M,g)$ is an Einstein, self-dual manifold.

Now, by \eqref{nablabochsymm}, \eqref{riemtwist} and the
local expression of $\nabla\ol{\operatorname{Riem}}$,
it is easy to compute, for instance,
\[
-\dfrac{2}{5}\ol{R}_{16,4}=
\ol{R}_{1556,4}=\dfrac{t^3S}{6912}
\pa{S^2-\dfrac{36S}{t^2}+\dfrac{288}{t^4}};
\]
since by \eqref{nablaricc}
\[
\ol{R}_{16,4}=-\dfrac{t^3S}{1728}
\pa{S^2-\dfrac{18S}{t^2}+\dfrac{72}{t^4}},
\]
we have the equality
\[
\dfrac{S}{6912}
\pa{S-\dfrac{12}{t^2}}\pa{S-\dfrac{24}{t^2}}=
\dfrac{S}{4320}
\pa{S-\dfrac{12}{t^2}}\pa{S-\dfrac{6}{t^2}}
\]
and it follows immediately that this holds if and only if
$S\in\{0,12/t^2\}$.

If we suppose $S=0$, by \eqref{riemtwist}
and \eqref{nablabochsymm}
it is easy to compute
\[
0=\ol{R}_{1334,6}=\dfrac{\ol{S}}{40}J_{3,6}^1=\dfrac{1}{20t^2}J_{3,6}^1
\]
which is impossible, since
\[
J_{3,6}^1=\dfrac{1}{2}t\sq{\dfrac{2}{t^2}-(\qq{14}+\qq{23})}=\dfrac{1}{t}
\neq 0.
\]
Thus, by Corollary \ref{kahlertwist} we conclude that
$(Z,g_t,J)$ is a K\"{a}hler-Einstein manifold and this ends the proof.

\end{proof}

\begin{rem} \label{bochnsymm}
\begin{enumerate}
\item It is worth to note that Theorem
\ref{complexlinricci} is a generalization
of a result due to Davidov, Grantcharov and Mu\v{s}karov, who showed that
the Riemann tensor $\ol{\operatorname{Riem}}$ of $(Z,g_t,J)$ satisfies
\begin{equation} \label{ah3identity}
	\ol{R}_{pqrs}=\ol{R}_{tuvw}J_p^tJ_q^uJ_r^vJ_s^w,
	\mbox{ } \forall 1\leq 	
	p,q,r,s\leq 6
\end{equation}
if and only if $(M,g)$ is an Einstein, self-dual manifold (see \cite{davmuscurv}).
Indeed,
a straightforward computation shows that \eqref{ah3identity} implies
\eqref{riccicommj}. Almost Hermitian manifolds which satisfy
\eqref{ah3identity} are sometimes called \emph{RK-manifolds}
see (\cite{vanhecke2} and \cite{vanyano}) and \eqref{ah3identity}
is a condition satisfied by every nearly K\"{a}hler manifold
(see \cite{graycurv}).

\item We point out that one can directly prove Corollary
\ref{bochflat} without exploiting Theorem \ref{nablabochflat};
indeed, if we suppose that $(Z,g_t,J)$ is a Bochner-flat manifold,
by \eqref{symmboch} we can show that $\ol{\operatorname{Riem}}$ is
a $K$-\emph{curvature-like tensor}, i.e. its components satisfy
\begin{equation} \label{kahlidcomp}
	\ol{R}_{pqts}J_r^t+\ol{R}_{pqrt}J_s^t=0,
	\mbox{ } \forall 1\leq 	
	p,q,r,s\leq 6.
\end{equation}
The equality in
\eqref{kahlidcomp} is sometimes referred to as \emph{K\"{a}hler identity}
and it is a deeply studied feature of almost Hermitian manifolds,
aside from the twistor spaces context (we may refer the reader to
\cite{graycurv}, \cite{sawaki}, \cite{sawaseki},
\cite{vanhecke} and \cite{vanyano}).

By another result due to Davidov, Grantcharov and Mu\v{s}karov
\cite{davmuscurv},
$\ol{\operatorname{Riem}}$ satisfies \eqref{kahlidcomp} if and only if
$(M,g)$ is an Einstein, self-dual manifold with $S\in\{0,12/t^2\}$
If $S=0$, by
\eqref{bochneralmherm}, \eqref{riemtwist}, \eqref{ricctwist} and
\eqref{scaltwist}, we obtain
\begin{align*}
	0=\ol{B}_{5656}=
	\dfrac{S}{20}+\dfrac{3}{10t^2}-\dfrac{9t^2}{40}\abs{\qt{ab}}^2=
	\dfrac{3}{10t^2},
\end{align*}
which is a contradiction. Thus, $(Z,g_t,J)$ is
K\"{a}hler-Einstein and, by Proposition \ref{bochflatkahl},
the claim is proved. For detailed dissertations about
K\"{a}hler, Bochner-flat
manifolds, see, for instance,
\cite{bryant}, \cite{chen}, \cite{chendillen} and \cite{chenyano}.
\end{enumerate}
\end{rem}

\section{Ricci parallel and locally symmetric twistor spaces} \label{kahleinst}

In this section, we discuss the case of a Riemannian four-manifold $(M,g)$
whose twistor space $(Z,g_t,J)$ is a K\"{a}hler-Einstein manifold.

Let us start proving a well-known result due to Friedrich and Grunewald
(see \cite{frigru}):
\begin{theorem} \label{einstwist}
Let $(M,g)$ a Riemannian four-manifold and $(Z,g_t)$ be its twistor
space. Then $(Z,g_t)$ is Einstein if and only if $(M,g)$ is Einstein,
self-dual with scalar curvature
$S$ equal to $6/t^2$ or to $12/t^2$.
\end{theorem}
\begin{proof}
Let us suppose that $M$ is Einstein, self-dual with
$S\in\{\frac{6}{t^2},\frac{12}{t^2}\}$. We obtain immediately
that $\ol{R}_{a5}=\ol{R}_{a6}=0$: indeed, since $M$ is Einstein,
then $(M,g)$ has a harmonic curvature metric, i.e.
$\diver\operatorname{Riem}=0$, which means that,
in particular, $\qt{ac}_{,c}=\qq{ac}_{,c}=0$ for every $a=1,...,4$.
We have that $M$ is Einstein and self-dual if and only if
\[
\begin{cases}
	\qd{13}=\qd{14}=\qd{23}=\qd{42}=0;\\
	\qt{12}=\qt{14}=\qt{23}=\qt{34}=0;\\
	\qq{12}=\qq{13}=\qq{42}=\qq{34}=0;\\
	\qt{13}=\qt{42}=\qq{14}=\qq{23}=\qd{12}=\qd{34}=\dfrac{S}{12}.
\end{cases}.
\]
Therefore, for instance we have that
\[
\overline{R}_{12}=-\frac{1}{2}t^2\sum_{c=1}^4({\qt{1c}\qt{2c}+\qq{1c}\qq{2c}})=0;
\]
with similar computations, we obtain that $\overline{R}_{ab}=0$ for every
$a\neq b$. Obviously, the system also implies that
$\overline{R}_{56}=0$. Now, let us consider
\begin{align*}
	\overline{R}_{11}&=\dfrac{S}{4}-\frac{1}{2}t^2\sum_{c=1}^4({\qt{1c}\qt{1c}+\qq{1c}\qq{1c}})=\\
	&=\dfrac{S}{4}-\frac{1}{2}t^2((\qt{13})^2+(\qq{14})^2)=\dfrac{S}{4}-t^2\dfrac{S^2}{144};
\end{align*}
we have that
\[
\overline{S}=S+\dfrac{2}{t^2}-t^2\dfrac{S^2}{72}
\]
Thus, we have
\[
\overline{R}_{11}=\dfrac{\overline{S}}{6}\Leftrightarrow
t^2S^2-18S+\dfrac{72}{t^2}=0\Leftrightarrow
S\in\left\{\dfrac{6}{t^2},\dfrac{12}{t^2}\right\},
\]
i.e. $\overline{R}_{11}=\overline{S}/6$ by hypothesis. With analogous
computations, we conclude that $\overline{R}_{aa}=\overline{S}/6$ for every $a$.
Finally,
\[
\overline{R}_{55}=\overline{R}_{66}=\dfrac{1}{t^2}+t^2\dfrac{S^2}{144}.
\]
Again, the right-hand side is equal to $\overline{S}/6$
if and only if $S$ is equal to $6/t^2$ or to $12/t^2$, which implies
that $Z$ is Einstein.

Conversely, let us suppose that $(Z,g_t)$ is an Einstein manifold, i.e.
$\ol{R}_{pq}=(\ol{S}/6)\delta_{pq}$. Recalling the expressions of $\ol{\operatorname{Riem}}$, $\ol{\ricc}$ and $\ol{\operatorname{S}}$
listed in \eqref{riemtwist}, \eqref{ricctwist} and \eqref{scaltwist},
respectively, we easily obtain
\begin{align*}
	\dfrac{S}{6}+\dfrac{1}{3t^2}-\dfrac{t^2}{24}\pa{\abs{\qt{ab}}^2+
		\abs{\qq{ab}}^2}&=\dfrac{\ol{S}}{6}=\ol{R}_{55}=
	\dfrac{1}{t^2}+\dfrac{t^2}{4}\abs{\qt{ab}}^2;\\
	\dfrac{S}{6}+\dfrac{1}{3t^2}-\dfrac{t^2}{24}\pa{\abs{\qt{ab}}^2+
		\abs{\qq{ab}}^2}&=\dfrac{\ol{S}}{6}=\ol{R}_{66}=
	\dfrac{1}{t^2}+\dfrac{t^2}{4}\abs{\qq{ab}}^2.
\end{align*}
This implies immediately that
\begin{align*}
	\dfrac{7t^2}{24}\abs{\qt{ab}}^2&=\dfrac{S}{6}-\dfrac{2}{3t^2}-
	\dfrac{t^2}{24}\abs{\qq{ab}}^2;\\
	\dfrac{7t^2}{24}\abs{\qq{ab}}^2&=\dfrac{S}{6}-\dfrac{2}{3t^2}-
	\dfrac{t^2}{24}\abs{\qt{ab}}^2.
\end{align*}
Subtracting the two equations, it is easy to show that
\[
\abs{\qt{ab}}^2=\abs{\qq{ab}}^2=\dfrac{S}{2t^2}-\dfrac{2}{t^4}.
\]
Now, by the expression of the components in \eqref{ricctwist},
we have that
\[
R_{ab}-\dfrac{t^2}{2}\sum_{c=1}^4\pa{\qt{ac}\qt{bc}+\qq{ac}\qq{bc}}=
\dfrac{\ol{S}}{6}\delta_{ab}=\pa{\dfrac{S}{8}+\dfrac{1}{2t^2}}\delta_{ab};
\]
a straightforward computation shows that
\[
R_{11}+R_{22}-R_{33}-R_{44}=
t^2\sq{(\qt{12})^2+(\qq{12})^2-(\qt{34})^2-(\qq{34})^2},
\]
that is,
\[
B_{11}=t^2(A_{12}B_{12}+A_{13}B_{13}),
\]
for every local orthonormal frame $e\in O(M)_-$. Thus, we can choose
a frame $e$ such that the associated matrix $A$ is diagonal, in order
to obtain $B_{11}=0$ (note that this is not a global condition).
By \eqref{transfabc}, if we us choose
\[
a_+=I_3, \qquad
a_-=
\pmat{
	0 & 1 & 0\\
	-1 & 0 & 0\\
	0 & 0 & 1
},
\]
we easily obtain that
\[
\tilde{B}=
\pmat{
	B_{21} & B_{22} & B_{23} \\
	0 & -B_{12} & -B_{13}\\
	B_{31} & B_{32} & B_{33}
},
\]
which leads to
\[
B_{21}=\tilde{B}_{11}=t^2(\tilde{A}_{12}\tilde{B}_{12}+
\tilde{A}_{13}\tilde{B}_{13})=0
\]
(note that also $\tilde{A}$ is a diagonal matrix). By analogous
computations, we can conclude that $B=0$, i.e. $(M,g)$ is an
Einstein manifold.

Now, by hypothesis we have that
\[
0=\tilde{R}_{56}=\frac{t^2}{4}\sum_{a,c=1}^4 \qt{ac}\qq{ac};
\]
since $(M,g)$ is Einstein, this equation assumes the form
\[
A_{12}A_{13}+A_{23}(A_{22}+A_{33})=0.
\]
Let us again choose $e\in O(M)_-$ such that $A$ is diagonal. Choosing
suitable matrices $a_+\in SO(3)$, it is easy to obtain that
\[
A_{11}=A_{22}\vee A_{11}=-A_{22}.
\]
If we suppose that $A_{11}\neq A_{22}$, analogous computations
immediately lead to a contradiction. Thus,
$A_{11}=A_{22}$. Exchanging $A_{22}$ and $A_{33}$, we can repeat
the same argument to obtain $A_{11}=A_{33}$. Thus, $(M,g)$
is self-dual.

Finally, since $(M,g)$ is Einstein and self-dual, we compute
\[
\dfrac{S^2}{36}=\abs{\qt{ab}}^2=\dfrac{S}{2t^2}-\dfrac{2}{t^4},
\]
that is,
\[
t^4S^2-18t^2S+72=0.
\]
Again, this equation holds if and only if
\[
S\in\left\{\dfrac{6}{t^2},\dfrac{12}{t^2}\right\}
\]
and this ends the proof.
\end{proof}

As a consequence, we can state another well known result (see \cite{mus} and \cite{cdm}):
\begin{cor} \label{kahlertwist}
	The twistor space $(Z,g_t,J)$ associated to $(M,g)$ is a
	K\"{a}hler-Einstein
	manifold if and only if $(M,g)$ is an Einstein, self-dual manifold with
	scalar curvature equal to $12/t^2$.
\end{cor}
\begin{proof}
	If $(M,g)$ is an Einstein, self-dual manifold with $S=12/t^2$, a direct
	inspection of the components in \eqref{ricctwist} and \eqref{nablajplus}
	shows that $(Z,g_t)$ is K\"{a}hler-Einstein. Conversely, by Theorem
	\ref{einstwist} we know that $(M,g)$ is Einstein, self-dual with $S$
	equal to $6/t^2$ or $12/t^2$. If we suppose $S=6/t^2$, we immediately
	obtain that
	\[
	J_{3,6}^1=-J_{4,5}^1=\dfrac{1}{2t}\neq 0,
	\]
	which contradicts the hypothesis that $(Z,g_t,J)$ is a K\"{a}hler
	manifold and ends the proof.
\end{proof}

We recall that compact, Einstein, self-dual four-manifolds with
positive scalar curvature have been classified:
indeed, Hitchin showed that there are
just two possibilites, up to conformal equivalences, which are $\mathds{S}^4$ and
$\cp^2$ with their canonical metrics \cite{hitch}.

Now, we want to provide an analogue of Proposition \ref{bochflatkahl}, in
order to give another characterization of $\mathds{S}^4$.
Indeed, by direct inspection of the components listed in \eqref{covriemkahl},
one can immediately show the validity of
\begin{theorem} \label{locsymmtwist}
Let $(Z,g_t,J)$ be a K\"{a}hler-Einstein twistor space. Then, $(Z,g_t,J)$ is
locally symmetric
if and only if $(M,g)$ is homotetically isometric to $\mathds{S}^4$ with its
canonical metric.
\end{theorem}

Note that, combining Theorem \ref{locsymmtwist} with Theorem
\ref{nablabochflat} and Corollary \ref{weylpartwist},
we can state the following
\begin{theorem} \label{spherecharact}
Let $(Z,g_t,J)$ be a K\"{a}hler-Einstein twistor space.
Then the following conditions are equivalent:
\begin{enumerate}
	\item $(Z,g_t,J)$ is a conformally symmetric manifold;
	\item $(Z,g_t,J)$ is a Bochner-parallel manifold;
	\item $(M,g)$ is homotetically isometric to $\mathds{S}^4$, with its
	canonical metric.
\end{enumerate}
\end{theorem}
\noindent
Recall that the equivalence $(2)\Leftrightarrow(3)$ holds without
any \emph{a priori} hypothesis on $(Z,g_t,J)$,
by Theorem \ref{nablabochflat}.

Moreover, the Einstein condition on $g_t$ implies that
$(Z,g_t,J)$ is a harmonic curvature manifold, i.e.
$\diver\ol{\operatorname{Riem}}\equiv 0$. Then, by this consideration,
equation \eqref{nablaboch} and Theorem \ref{spherecharact} we
can state the following
\begin{theorem}
Let $(Z,g_t,J)$ be a K\"{a}hler-Einstein twistor space. Then, either
$(Z,g_t,J)$ is Bochner-parallel (and, then,
$(M,g)$ is homotetically isometric to $\mathds{S}^4$)
or $\nabla{\ol{\operatorname{B}}}\neq 0$ and
$\diver\ol{\operatorname{B}}\equiv 0$.
\end{theorem}
Though in Theorem \ref{nablabochflat}
we characterized $\mathds{S}^4$ as the only four-dimensional
manifold whose twistor space is Bochner-parallel, we cannot obtain
an analogous characterization if we drop the hypothesis
of $g_t$ being K\"{a}hler-Einstein
in Theorem \ref{locsymmtwist}. For instance, by \eqref{riemtwist},
an easy computation of the
local expression of $\nabla{\operatorname{Riem}}$ proves the following
\begin{proposition} \label{ricflatselfdual}
Let $(M,g)$ be a Ricci-flat, self-dual, locally symmetric
four-manifold. Then, its twistor space $(Z,g_t)$ is locally symmetric.
\end{proposition}
%
Even though $\mathds{S}^4$ is not the only four-manifold whose twistor space
satisfies $\nabla\ol{\operatorname{Riem}}\equiv 0$, if we consider
Hermitian twistor spaces, i.e. the ones associated to self-dual
manifolds (see \cite{athisin} and \cite{cdm}), we can state the following
\begin{theorem} \label{locsymmricpar}
Let $(M,g)$ be a self-dual Riemannian manifold and let $(Z,g_t)$ be
its twistor space. Then,
\begin{enumerate}
	\item $(Z,g_t)$ is Ricci parallel
	(i.e., $\nabla\ol{\ricc}\equiv 0$) if and only if
	$(Z,g_t)$ is an Einstein manifold or $(M,g)$ is Ricci-flat;
	\item $(Z,g_t)$ is locally symmetric if and only if
	$(M,g)$ is homotetically isometric to $\mathds{S}^4$
	with its canonical metric
	or $(M,g)$ is Ricci-flat and locally symmetric.
\end{enumerate}
\end{theorem}
\begin{proof}
\begin{enumerate}
	\item
	First, let us suppose that $(M,g)$ is an Einstein,
	self-dual manifold with scalar curvature $S$. In particular,
	$(M,g)$ is Ricci parallel.
	Moreover, since under these hypotheses
	\[
	\qd{de}\qt{de}=\qd{de}\qq{de}=\qt{de}\qq{de}
	=\qd{ab}_{,c}=\qt{ab}_{,c}=\qq{ab}_{,c}=0
	\]
	for
	every $a,b,c$, we immediately obtain
	that the only components listed in \eqref{nablaricc} that may not
	vanish are $\ol{R}_{ab,5}, \ol{R}_{ab,6}, \ol{R}_{a5,b},
	\ol{R}_{a6,b}$ for some $a,b$. A straightforward computation
	shows that $\ol{R}_{ab,5}=\ol{R}_{ab,6}=0$, regardless of the
	value of $S$.
	
	Let us consider, for instance,
	\[
	\ol{R}_{15,b}=-\dfrac{t^3 S}{1728}\pa{S^2-\dfrac{18S}{t^2}
		+\dfrac{72}{t^4}}\delta_{3b}=
	-\dfrac{t^3 S}{1728}\pa{S-\dfrac{12}{t^2}}\pa{S-\dfrac{6}{t^2}}
	\delta_{3b}.
	\]
	Similar
	computations show that the other components of
	$\nabla{\ol{\ricc}}$ vanish for the same values of $S$.
	Thus, by Theorem
	\ref{einstwist}, we conclude that, if $(M,g)$ is an Einstein,
	self-dual manifold, $(Z,g_t)$ is Ricci parallel if and only if
	it is an Einstein manifold or the scalar curvature of $(M,g)$
	vanishes.
	
	Thus, in order to prove the statement, it is sufficient to show
	that any self-dual manifold whose twistor space is Ricci parallel
	must also be an Einstein manifold. Under these hypotheses, we can
	observe that
	\[
	0=\ol{R}_{55,6}=\ol{R}_{66,5}\Longra
	0=\qd{ab}\qt{ab}=\qd{ab}\qq{ab};
	\]
	by the self-duality condition, these equations become
	\[
	\begin{cases}
		B_{13}B_{11}+B_{31}B_{33}+B_{21}B_{23}=0\\
		B_{12}B_{11}+B_{21}B_{22}+B_{31}B_{32}=0
	\end{cases} .
	\]
	Since the system holds for every $e\in O(M)_-$, by a suitable change
	of frames we obtain that
	\[
	B_{12}B_{13}+B_{22}B_{23}+B_{32}B_{33}=0.
	\]
	The validity of these three equations is equivalent to
	the orthogonality of the columns of the matrix $B$; this means
	that $B^T B=D$, where $D$ is a diagonal matrix. Moreover, let us		denote the i-th column of $B$ as $v_i$ and let us define
	\[
	||v_i||^2:=\sum_{j=1}^3 B_{ji}B_{ji};
	\]
	furthermore, let us suppose that $v_i\neq 0$ for every $i$.
	Since the columns of $B$ are the rows of $B^T$, if we replace
	the rows of $B^T$ with $v_i^T/||v_i||^2$, we obtain that
	the new matrix, which we denote as $B_{or}$, is orthogonal.
	In particular, we may assume that $B_{or}\in SO(3)$ (otherwise,
	we could replace it with $-B_{or}$); thus, putting
	$a_-^{-1}=B_{or}$ and $a_+=I_3$ in \eqref{transfabc},
	we have that
	\[
	a_-^{-1}Ba_+=B_{or}B=I_3.
	\]
	By the expressions of the entries of the matrices
	$A$ and $B$, it is easy to obtain
	\begin{align*}
		\ol{R}_{13,5}&=2t\pa{\dfrac{S}{3}+2}-\dfrac{t^3 S}{12}=0;\\
		\ol{R}_{24,5}&=\dfrac{tS}{6}-\dfrac{t^3 S}{12}
		\pa{\dfrac{S}{6}-1}=0.
	\end{align*}
	The two equations can hold simultaneously if and only if
	$t=\sqrt{3+\sqrt{17}}$, which implies that
	\[
	S=\dfrac{3(1+\sqrt{17})}{2}.
	\]
	However, since $S$ is invariant under
	change of frames, we can choose
	\[
	a_-=\pmat{
		0 & 1 & 0\\
		-1 & 0 & 0\\
		0 & 0 & 1
	}, \qquad a_+=I_3
	\]
	in \eqref{transfabc} in order to obtain
	\[
	B=\pmat{
		0 & -1 & 0\\
		1 & 0 & 0\\
		0 & 0 & 1
	};
	\]
	now, it is easy to compute
	\[
	\ol{R}_{12,5}=-t+\dfrac{t^3 S}{12}=0,
	\]
	which holds if and only if
	\[
	S=\dfrac{12}{t^2}=\dfrac{3(\sqrt{17}-3)}{2}\neq
	\dfrac{3(1+\sqrt{17})}{2}.
	\]
	Therefore, we get a contradiction and we conclude that at least
	one of the columns of $B$ must be made of zeros for every
	$e\in O(M)_-$, which obviously implies that $(M,g)$ is an Einstein
	manifold.
	\item We know that $\mathds{S}^4$ and every locally symmetric,
	self-dual, Ricci-flat
	manifold have locally symmetric twistor spaces, by Proposition
	\ref{ricflatselfdual}. Conversely, let us suppose that
	$(Z,g_t)$ is locally symmetric: in particular, it is Ricci parallel,
	therefore $(Z,g_t)$ is an Einstein manifold or $(M,g)$ is Ricci-flat.
	
	If $S\neq 0$, by \eqref{riemtwist}, it is easy to compute
	\begin{align*}
		\ol{R}_{5653,u}\theta^u&=
		\dfrac{t}{2}\left\{\sq{-\dfrac{3}{2}\qd{34}+\dfrac{t^2}{2}
			\qq{23}\qt{42}+\dfrac{t^2}{4}\qt{13}\qq{14}}\qt{42}-
		\dfrac{t^3}{8}(\qt{13})^2\qq{23}+\dfrac{1}{2t}\qq{23}\right\}
		\theta^2=\\
		&=\dfrac{t^3S}{6912}\pa{S^2-\dfrac{36S}{t^2}+\dfrac{288}{t^4}}=
		\dfrac{t^3S}{6912}\pa{S-\dfrac{12}{t^2}}\pa{S-\dfrac{24}{t^2}}=0,
	\end{align*}
	which holds if and only if $S\in\{12/t^2,24/t^2\}$.
	Since $(Z,g_t)$ must be an Einstein manifold, this implies that
	$S\in\{6/t^2,12/t^2\}$: therefore, $S=12/t^2$ and, by Theorem
	\ref{locsymmtwist}, we conclude that $(M,g)$ is a spherical space form.
	
	Finally, let us suppose that $S=0$, i.e. $(M,g)$ is Ricci-flat.
	Since
	\[
	\ol{R}_{abcd,e}=R_{abcd,e}, \mbox { for every } a,b,c,d,e=1,...,4,
	\]
	it is apparent that $(M,g)$ is locally symmetric by hypothesis and this
	ends the proof.
\end{enumerate}
\end{proof}



\section{A quadratic formula for $\nabla J$ and higher order conditions}
%
%

\subsection{General quadratic formula for the square norm of $\nabla J$}

We begin stating a general result:

\begin{theorem} \label{squarejtwist}
Let $(M,g)$ be a Riemannian four-manifold and $(Z,g_t,J)$ be its twistor
space. Then, the equality
\begin{equation} \label{squarenablaj}
	\abs{\nabla J}^2=\dfrac{1}{3}\abs{d\omega}^2+\dfrac{1}{8}\abs{N_J}^2
\end{equation}
holds, where $\abs{d\omega}^2=\sum_{p,q,t=1}^6 d\omega(e_p,e_q,e_t)d\omega(e_p,e_q,e_t)$ and
$\abs{N_J}^2=\sum_{p,q,t=1}^6 N_{pq}^tN_{pq}^t$.
\end{theorem}
\begin{proof}
By direct computation, we obtain
\begin{align*}
	\abs{d\omega}^2=6t^2&\left[\qq{12}^2+\qt{12}^2+\qq{13}^2+\pa{\qt{13}-
		\dfrac{1}{t^2}}^2+\pa{\qq{14}-
		\dfrac{1}{t^2}}^2+\right. \\
	&+\qt{14}^2+\pa{\qq{23}-
		\dfrac{1}{t^2}}^2+\qt{23}^2+\qq{42}^2+\pa{\qt{42}-
		\dfrac{1}{t^2}}^2+\\
	&+\left.\qq{34}^2+\qt{34}^2\vphantom{\pa{\qq{14}-
			\dfrac{1}{t^2}}^2}\right];
\end{align*}
\begin{align*}
	\abs{N_J}^2=8t^2[(\qt{13}+\qt{42}-\qq{14}-\qq{23})^2+
	(\qt{14}+\qt{23}+\qq{13}+\qq{42})^2].
\end{align*}
Comparing these expressions with \eqref{squarejplus}, it is easy to obtain
\eqref{squarenablaj}.
\end{proof}

It is worth to point out that Theorem \ref{squarejtwist} allows to give
alternate proofs of some well-known results due to Mu\v{s}karov (see
\cite{mus}), exploiting the quadratic
relations among the invariants listed by
Gray and Hervella (see \cite{grayher}). For instance, we can reformulate the
following

\begin{proposition}
Let $(M,g)$ be a Riemannian manifold and $(Z,g_t,J)$ be its twistor space. Then,
\[
(Z,g_t,J)\in\NK\cup\AK\Lra(Z,g_t,J)\in\K,
\]
where $\NK$, $\AK$ and $\K$ denote the classes of nearly K\"{a}hler, almost
K\"{a}hler and K\"{a}hler manifolds, respectively.
\end{proposition}
\begin{proof}
One direction is trivial. Let us
suppose $Z\in\NK$. Then, by table IV in \cite{grayher}, we know that
\[
\abs{\nabla J}^2=\dfrac{1}{9}\abs{d\omega}^2;
\]
inserting this equation in \eqref{squarenablaj}, it is easy to obtain
\[
\abs{d\omega}^2=\abs{N_J}^2=0\Ra\abs{\nabla J}^2=0,
\]
that is, $Z\in K$. Now, let us suppose $Z\in\AK$. By the same table,
we have that
\[
\abs{\nabla J}^2=\dfrac{1}{4}\abs{N_J}^2\mbox{ and } \abs{d\omega}^2=0;
\]
again, inserting these equations in \eqref{squarenablaj}, we obtain
$\abs{\nabla J}^2=0$, i.e. $Z\in\K$.
\end{proof}

In fact, by analogous calculations, we can prove more.
Let us consider the sixteen classes of almost
Hermitian manifolds listed in \cite{grayher}. Then, by Theorem
\ref{squarejtwist}, we can obtain an alternate proof of the following
statement, which was first proven by Mu\v{s}karov (see \cite{mus}):
\begin{theorem}
Let $(M,g)$ be a Riemannian four-manifold and $(Z,g_t,J)$ be its twistor
space. If $(Z,g_t,J)$ belongs to one of the first fifteen classes of almost
Hermitian manifolds, then $(Z,g_t,J)\in\H$; consequently, $(M,g)$ is
self-dual.
\end{theorem}

\subsection{Laplacian of the almost complex structures}

In this section, we consider the Laplacian $\Delta_JJ$
of the almost complex structure $J$ (for the definition and the
components, see \ref{appendc}). We say that
$J$ is \textit{harmonic} if $\Delta_JJ=0$ (see also
\cite{wood} and \cite{wood2}). By a direct inspection of the components
listed in \ref{appendc}, we can provide an alternate proof to a
well-known result, due to Davidov and Mu\v{s}karov (see \cite{davmus2}):
\begin{theorem}
Let $(M,g)$ be a Riemannian four-manifold and $(Z,g_t,J)$ be its
twistor space. Then, $J$ is harmonic if and only if $(M,g)$ is
self-dual.
\end{theorem}
\begin{proof}
One direction is trivial; indeed, since $(M,g)$ is self-dual,
\begin{align*}
	\Delta_JJ_3^1&=2A_{12}+\dfrac{1}{2}t[N_{14}^5A_{13}-
	N_{13}^5A_{12}]=0;\\
	\Delta_JJ_4^1&=2A_{13}+\dfrac{1}{2}t[N_{14}^5A_{12}-
	N_{13}^5A_{13}]=0,
\end{align*}
where the $N_{pq}^r$ are the components of the Nijenhuis tensor
of $J$ (see \eqref{nijcompplus}).
Moreover, as an immediate consequence of the self-duality condition and
the second Bianchi identity, by \eqref{weyldiv} and \eqref{laplweyldiv}
it is easy to show that
\[
\Delta_JJ_5^1=\Delta_JJ_6^1=\Delta_JJ_5^3=\Delta_JJ_6^3=0.
\]
Conversely, let us suppose $\Delta_JJ=0$. By the explicit expression
of $\Delta_JJ_v^u$ and $N_{pq}^r$ listed in \eqref{lapltwist} and in
\eqref{nijcompplus}, respectively, we obtain the global equations
\begin{align} \label{eqlaplself}
	2A_{12}-2t^2A_{23}A_{13}-t^2A_{12}(A_{33}-A_{22})&=0;\\
	2A_{13}-2t^2A_{23}A_{12}-t^2A_{13}(A_{33}-A_{22})&=0. \notag
\end{align}
Again, let us choose a local orthonormal frame $e\in O(M)_-$ such that
$A$ is a diagonal matrix. By the transformation law for the matrix
$A$ defined in \eqref{transfabc}, choosing the matrix
\[
a_+=
\pmat{
	\dfrac{1}{\sqrt{2}} & -\dfrac{1}{\sqrt{2}} & 0\\
	\dfrac{1}{\sqrt{2}} & \dfrac{1}{\sqrt{2}} & 0\\
	0 & 0 & 1
}
\]
and using the equations \eqref{eqlaplself}, we obtain
\[
(A_{22}-A_{11})\sq{\dfrac{2}{t^2}-A_{33}+\dfrac{1}{2}(A_{11}+A_{22})}=0.
\]
Let us suppose that $A_{11}\neq A_{22}$; this implies immediately that
\[
A_{33}=\dfrac{1}{2}(A_{11}+A_{22})+\dfrac{2}{t^2}.
\]
With similar computations, it is easy to show that, if this equation
holds, by \eqref{eqlaplself} we must have
\[
A_{11}-A_{22}=\dfrac{4}{t^2}\vee A_{11}-A_{22}=-\dfrac{4}{3t^2}.
\]
In both cases, choosing the matrix
\[
a_+=
\pmat{
	\dfrac{1}{2} & -\dfrac{\sqrt{3}}{2} & 0\\
	\dfrac{\sqrt{3}}{2} & \dfrac{1}{2} & 0\\
	0 & 0 & 1
}
\]
and applying the transformation law for $A$, the equations
\eqref{eqlaplself} lead to
\[
\dfrac{\sqrt{3}}{t^2}=0\vee \dfrac{5}{3\sqrt{3}t^2}=0,
\]
which clearly are contradictions. Thus, $A_{11}=A_{22}$; by
analogous computations, we easily obtain $A_{22}=A_{33}$.
Therefore, we can conclude that $(M,g)$ is self-dual.
\end{proof}

\subsection{Nijenhuis tensors of $J$ and $\JJ$}

Now, let us consider
the Nijenhuis tensors $N_J$ and $N_{\JJ}$
associated to $J$ and $\JJ$ respectively. The expression of the components
$J_{q,t}^p$ and $\JJ_{q,t}^p$ of the covariant derivatives of $J$ and $\JJ$
are listed in \eqref{nablajplus} and \eqref{nablajmin},
while the components of the tensors $N_J$ and
$N_{\JJ}$ are listed in \eqref{nijcompplus} and \eqref{nijcompmin},
together with the components of
their covariant derivatives and their divergences in
\eqref{nablanijplus} and \eqref{nablanijmin}. Thus, let us
recall the local expression of the divergences of $N_J$ with
respect to a local orthonormal coframe and its dual frame:
\begin{equation} \label{nijdiver}
\diver N_J=N_{pq,t}^t\theta^p\otimes\theta^q \qquad
\ol{\diver} N_J=N_{pt,t}^r\theta^p\otimes e_r,
\end{equation}
where
$\nabla N_J=N_{tq,s}^p\theta^s\otimes\theta^t\otimes\theta^q\otimes e_p$.
The following Theorem
is a generalization of a result due to Atiyah, Hitchin and Singer, which
characterizes self-dual four-manifolds as the ones whose twistor space
is Hermitian with respect to $J$ (see \cite{athisin} and \cite{cdm}):
\begin{theorem} \label{divnijzero}
Let $(M,g)$ be a Riemannian four-manifold and
$(Z,g_t,J)$ be its twistor space.
Then, $(M,g)$ is self-dual if and only if
$\diver{N_J}\equiv0\vee\ol{\diver}{N_J}\equiv0$.
\end{theorem}
\begin{proof}
One direction is trivial: indeed, we know that, if $(M,g)$ is self-dual,
$N_J\equiv 0$. Therefore, it is easy to see that
$\diver{N_J}=\ol{\diver}{N_J}\equiv0$.

Conversely, let us suppose that $\diver N_J\equiv 0$. By
\eqref{nijcompplus} and \eqref{divnijplus1}, we have
\[
\begin{cases}
	A_{13}\pa{A_{33}-A_{22}-\dfrac{4}{t^2}}+2A_{23}A_{12}&=0\\
	A_{12}\pa{A_{33}-A_{22}+\dfrac{4}{t^2}}-2A_{23}A_{13}&=0
\end{cases}.
\]
Let us choose a frame $e\in O(M)_-$ such that $A$ is diagonal. By
\eqref{transfabc}, choosing
\[
a_+=
\pmat{
	\frac{1}{\sqrt{2}} & -\frac{1}{\sqrt{2}} & 0\\
	\frac{1}{\sqrt{2}} & \frac{1}{\sqrt{2}} & 0\\
	0 & 0 & 1
},
\]
for the transformed matrix we obtain
\[
(A_{22}-A_{11})\sq{\dfrac{1}{2}(A_{11}+A_{22})-\dfrac{4}{t^2}-A_{33}}=0,
\]
that is,
\[
A_{11}=A_{22} \vee A_{33}=\dfrac{1}{2}(A_{11}+A_{22})-\dfrac{4}{t^2}.
\]
Let us suppose $A_{11}\neq A_{22}$. By choosing
\[
a_+=
\pmat{
	\frac{1}{2} & -\frac{\sqrt{3}}{2} & 0\\
	\frac{\sqrt{3}}{2} & \frac{1}{2} & 0\\
	0 & 0 & 1
}
\]
and applying again \eqref{transfabc}, it is easy to see that
$A_{11}=A_{22}$, which is a contradiction. Thus, with respect to $e$
we have that $A_{11}=A_{22}$. By a similar argument, one can easily show
that $A_{11}=A_{33}$, i.e. $(M,g)$ is self-dual.

Now, let us consider the case $\ol{\diver}N_J\equiv0$ and let us
choose again a frame $e\in O(M)_-$ such that $A$ is diagonal.
Then, since by hypothesis
\[
N_{1t,t}^1+N_{2t,t}^2=0,
\]
we can rewrite this equation as
\[
(A_{33}-A_{22})^2=0,
\]
which means that $A_{22}=A_{33}$. The equation above holds for
every frame with respect to which the matrix $A$ is diagonal;
thus, similarly we can obtain $A_{11}=A_{22}$, i.e.
$(M,g)$ is self-dual.
\end{proof}

As a consequence, it is immediate to show the validity of
\begin{cor} \label{nablanijzero}
	$(M,g)$ is self-dual if and only if $\nabla N_J\equiv 0$.
\end{cor}
We point out that the equation $\nabla N_J\equiv 0$ has been studied
in the wider context of almost Hermitian manifolds: for instance, Vezzoni
showed that any almost K\"{a}hler manifold that satisfies this condition
is, in fact, a K\"{a}hler manifold (see \cite{vezz}).

Note that, on the contrary, a simple inspection of the coefficients
listed in \eqref{nablanijmin} shows that $\nabla N_{\JJ}$ never vanishes
(the fact that $\JJ$ is never integrable was first proven by Eells and
Salamon \cite{elsal} and it is apparent by \eqref{nijcompmin}; we also
mention that, if $(M,g)$ is Einstein and self-dual,
$N_{\JJ}$ is parallel with respect to the
Chern connection $\tilde{\nabla}$ defined on $(Z,g_t,\JJ)$, as
shown by Davidov, Grantcharov and Mu\v{s}karov \cite{davmuschern}); thus,
we cannot obtain an analogue of Corollary \ref{nablanijzero} for
$N_{\JJ}$.
However, we can consider the divergences $\diver N_{\JJ}$ and
$\ol{\diver}N_{\JJ}$ of $N_{\JJ}$ (the components are listed in
\eqref{divnijmin1} and \eqref{divnijmin2}) and state the following
\begin{theorem}
Let $(M,g)$ be a four-dimensional, self-dual Riemannian manifold.
If $M$ is Ricci-flat, then $\ol{\diver}N_{\JJ}=0$. If the scalar curvature
$S$ of $(M,g)$ is different from $6/t^2$, then the converse holds.
\end{theorem}
\begin{proof}
Suppose that $(M,g)$ is Ricci-flat, i.e. $\ricc=0$. In particular, $(M,g)$
is Einstein, self-dual with $S=0$. This implies that
\[
\eqsist{
	\qt{12}&=\qt{14}=\qt{23}=\qt{34}=\qq{12}=\qq{13}=\qq{42}=\qq{34}=0;\\
	\qt{13}&=\qt{42}=\qq{14}=\qq{23}=0.
}.
\]
In particular, by the second equation we obtain
\[
\Sigma=0\Rightarrow (\Sigma)_{,a}=0, \mbox{ } \forall a=1,...,4.
\]
Therefore, $\ol{\diver}N_{\JJ}=0$, by direct inspection.

Now, let us suppose that $\ol{\diver}N_{\JJ}=0$ and $S\neq 6/t^2$. It is easily shown that the hypothesis on the scalar curvature leads to
\[
\Sigma-\dfrac{2}{t^2}\neq 0, \mbox{ on } O(M)_-.
\]
Indeed, if $\Sigma-\dfrac{2}{t^2}=0$, the matrix $A$ appearing in the decomposition of the Riemann curvature operator has the form
\[
A=
\pmat{
	\frac{S}{12} & 0 & 0\\
	0 & \frac{S}{12} & 0\\
	0 & 0 & \frac{1}{t^2}-\frac{S}{12}
}
\]
for every local orthonormal frame, since $M$ is self-dual. Then, we must have
\[
\frac{S}{12}=\frac{1}{t^2}-\frac{S}{12} \Leftrightarrow S=\dfrac{6}{t^2},
\]
which is a contradiction. Thus, by hypothesis, we must have
\begin{align*}
	\qt{13}+\qq{14}=
	\qt{13}+\qq{23}&=
	\qt{42}+\qq{23}=
	\qt{42}+\qq{14}=
	\qt{13}+\qt{42}=
	\qq{14}+\qq{23}=0\\
	\qt{12}&=\qt{34}=\qq{12}=\qq{34}=0.
\end{align*}
These equations immediately imply that $(M,g)$ is Einstein with $S=0$, i.e. $M$ is Ricci-flat.
\end{proof}

\appendix

\section{Divergence of the self-dual part of the Weyl tensor} \label{appweyl}

We list the components of the divergence of $W^+$ in a Riemannian
four-manifold.

\[
W^+_{abcd, e}\omega^e = dW^+_{abcd}-W^+_{ebcd}\omega^e_a-W^+_{aecd}\omega^e_b-W^+_{abed}\omega^e_c-W^+_{abce}\omega^e_d;
\]
\[
\pa{\delta W^+}_{abc} = \sum_{e=1}^4 W^+_{cbae, e}= \sum_{e=1}^4 W^+_{eabc, e}; \quad \pa{\delta W^+}_{abc}=-\pa{\delta W^+}_{acb}.
\]
By \eqref{riemq} and \eqref{covderq}, we obtain:
\small
\noindent
\begin{multicols}{2}
\noindent
\begin{align} \label{weyldiv}
\pa{\delta W^+}_{121} &= \frac{1}{4}\pa{\qd{12}_{, 2}+\qd{13}_{, 3}+\qd{14}_{, 4}} - \frac{1}{24}S_{2}; \\
\pa{\delta W^+}_{131} &= \frac{1}{4}\pa{\qt{12}_{, 2}+\qt{13}_{, 3}+\qt{14}_{, 4}} - \frac{1}{24}S_{3}; \notag\\
\pa{\delta W^+}_{141} &= \frac{1}{4}\pa{\qq{12}_{, 2}+\qq{13}_{, 3}+\qq{14}_{, 4}} - \frac{1}{24}S_{4}; \notag
\end{align}
\begin{align*}
\pa{\delta W^+}_{212} &= \frac{1}{4}\pa{\qd{12}_{, 1}-\qd{23}_{, 3}+\qd{42}_{, 4}} - \frac{1}{24}S_{1}; \notag\\
\pa{\delta W^+}_{213} &= \frac{1}{4}\pa{\qt{12}_{, 1}-\qt{23}_{, 3}+\qt{42}_{, 4}} - \frac{1}{24}S_{4}; \notag
\end{align*}
\begin{align*}
\pa{\delta W^+}_{214} &= \frac{1}{4}\pa{\qq{12}_{, 1}-\qq{23}_{, 3}+\qq{42}_{, 4}} + \frac{1}{24}S_{3}; \notag\\
\pa{\delta W^+}_{312} &= \frac{1}{4}\pa{\qd{13}_{, 1}+\qd{23}_{, 2}-\qd{34}_{, 4}} + \frac{1}{24}S_{4}; \notag\\
\pa{\delta W^+}_{313} &= \frac{1}{4}\pa{\qt{13}_{, 1}+\qt{23}_{, 2}-\qt{34}_{, 4}} - \frac{1}{24}S_{1}; \notag\\
\pa{\delta W^+}_{314} &= \frac{1}{4}\pa{\qq{13}_{, 1}+\qq{23}_{, 2}-\qq{34}_{, 4}} - \frac{1}{24}S_{2}; \notag
\end{align*}
\begin{align*}
\pa{\delta W^+}_{412} &= \frac{1}{4}\pa{\qd{14}_{, 1}-\qd{42}_{, 2}+\qd{34}_{, 3}} - \frac{1}{24}S_{3}; \notag\\
\pa{\delta W^+}_{413} &= \frac{1}{4}\pa{\qt{14}_{, 1}-\qt{42}_{, 2}+\qt{34}_{, 3}} + \frac{1}{24}S_{2}; \notag\\
\pa{\delta W^+}_{414} &= \frac{1}{4}\pa{\qq{14}_{, 1}-\qq{42}_{, 2}+\qq{34}_{, 3}} - \frac{1}{24}S_{1}.\notag
\end{align*}
\end{multicols}
\normalsize
\section{Riemann curvature of a twistor space} \label{appa}

We recall here all the components of the Riemann tensor, the Ricci tensor and the scalar curvature of the twistor space for a Riemannian four-manifold (see also \cite{jenrig}).
\\ \\
\textbf{Components of the Riemann curvature tensor} $\overline{\operatorname{Riem}}$ on $(Z,g_t)$:
\footnotesize
\noindent
\begin{multicols}{2}
\noindent
\begin{align} \label{riemtwist}
\overline{R}_{abcd} &= R_{abcd} -\frac{1}{4}t^2\left[\pa{\qt{ac}\qt{bd}-\qt{ad}\qt{bc}}\right.\\ \nonumber
&\left.+\pa{\qq{ac}\qq{bd}-\qq{ad}\qq{bc}}\right] \\ \nonumber &-\frac{1}{2}t^2\pa{\qt{ab}\qt{cd}+\qq{ab}\qq{cd}};\\ \nonumber
\overline{R}_{ab56} &= \qd{ab} -\frac{1}{4}t^2\sum_{c=1}^4\pa{\qt{ac}\qq{bc}-\qt{bc}\qq{ac}}; \nonumber\\
\overline{R}_{abc5} &= -\frac{1}{2}t \pa{\qt{ab}}_{c}; \qquad \overline{R}_{abc6} = -\frac{1}{2}t \pa{\qq{ab}}_{c};\nonumber
\end{align}
\begin{align*}
\overline{R}_{5656} &= \frac{1}{t^2}; \quad \overline{R}_{565b}=\overline{R}_{56a6}=0; \nonumber\\
\overline{R}_{5ab5} &= -\frac{1}{4}t^2\sum_{c=1}^4\qt{ac}\qt{bc}; \nonumber\\
\overline{R}_{5ab6} &= -\frac{1}{2}\qd{ab}-\frac{1}{4}t^2\sum_{c=1}^4\qt{bc}\qq{ac};\nonumber\\
\overline{R}_{6ab6} &= -\frac{1}{4}t^2\sum_{c=1}^4\qq{ac}\qq{bc}; \nonumber
\end{align*}
\begin{align*}
\overline{R}_{6ab5} &= \frac{1}{2}\qd{ab}-\frac{1}{4}t^2\sum_{c=1}^4\qq{bc}\qt{ac}.
\end{align*}
\end{multicols}
\normalsize
\textbf{Components of the Ricci tensor and}  $\overline{\ricc}$ \textbf{and
	of the scalar curvature} $\ol{S}$:
\footnotesize
\noindent
\begin{multicols}{2}
\noindent
\begin{align} \label{ricctwist}
\overline{R}_{ab} &= R_{ab} -\frac{1}{2}t^2\sum_{c=1}^4\pa{\qt{ac}\qt{bc}+\qq{ac}\qq{bc}}; \\ \overline{R}_{a5} &= \frac{1}{2}t\sum_{c=1}^4\pa{\qt{ac}}_c; \notag\\ \overline{R}_{a6} &= \frac{1}{2}t\sum_{c=1}^4\pa{\qq{ac}}_c; \notag
\end{align}
\begin{align*}
\overline{R}_{55} &= \frac{1}{t^2} + \frac{t^2}{4}\abs{\qt{ab}}^2; \notag \\ \overline{R}_{56} &= \frac{1}{4}t^2\sum_{a,c=1}^4 \qt{ac}\qq{ac}; \notag \\ \overline{R}_{66} &= \frac{1}{t^2} + \frac{t^2}{4}\abs{\qq{ab}}^2, \notag
\end{align*}
\end{multicols}
\begin{multicols}{1}
\begin{align} \label{scaltwist}
\overline{S} &= S + \frac{2}{t^2}- \frac{1}{4}t^2\pa{\abs{\qt{ab}}^2+\abs{\qq{ab}}^2}
\end{align}
\end{multicols}
\normalsize
(here: $\abs{\qt{ab}}^2 = \sum_{a, b=1}^4\pa{\qt{ab}}^2$, and similarly for $\qq{ab}$).

\vspace{0.5cm}

\textbf{Components of the Weyl tensor} $\overline{\operatorname{W}}$:
\footnotesize
\begin{align} \label{weyltwist}
\overline{W}_{abcd}&=R_{abcd}-\dfrac{1}{4}t^2[(\qt{ac}\qt{bd}-\qt{ad}\qt{bc})+
(\qq{ac}\qq{bd}-\qq{ad}\qq{bc})]
-\dfrac{1}{2}t^2(\qt{ab}\qt{cd}+\qq{ab}\qq{cd})\\+
&-\dfrac{1}{4}\Bigg\{
\left[R_{ac}-\dfrac{1}{2}t^2(\qt{ae}\qt{ce}+\qq{ae}\qq{ce})\right]\delta_{bd}
-\left[R_{bc}-\dfrac{1}{2}t^2(\qt{be}\qt{ce}+\qq{be}\qq{ce})\right]\delta_{ad}+\notag\\
&+\left[R_{bd}-\dfrac{1}{2}t^2(\qt{be}\qt{de}+\qq{be}\qq{de})\right]\delta_{ac}-\left[R_{ad}-\dfrac{1}{2}t^2(\qt{ae}\qt{de}+\qq{ae}\qq{de})\right]\delta_{bc}\Bigg\}+
\dfrac{\overline{S}}{20}(\delta_{ac}\delta_{bd}-\delta_{ad}\delta_{bc}); \notag
\end{align}
	\begin{align*}
		\overline{W}_{5ab5}&=\dfrac{1}{4}R_{ab}-\dfrac{3}{8}t^2\qt{ac}\qt{bc}-\dfrac{1}{8}t^2\qq{ac}\qq{bc}+
		\left(\dfrac{3}{20t^2}+\dfrac{3}{40}t^2|\qt{cd}|^2+\dfrac{1}{80}t^2|\qq{cd}|^2-\dfrac{S}{20}\right)\delta_{ab}; \notag\\
		\overline{W}_{6ab6}&=\dfrac{1}{4}R_{ab}-\dfrac{3}{8}t^2\qq{ac}\qq{bc}-\dfrac{1}{8}t^2\qt{ac}\qt{bc}+
		\left(\dfrac{3}{20t^2}+\dfrac{3}{40}t^2|\qq{cd}|^2+\dfrac{1}{80}t^2|\qt{cd}|^2-\dfrac{S}{20}\right)\delta_{ab}; \notag\\
\overline{W}_{abc5}&=-\dfrac{1}{2}t\left(\qt{ab}_{,c}+\dfrac{1}{4}\qt{bd}_{,d}\delta_{ac}-
\dfrac{1}{4}\qt{ad}_{,d}\delta_{bc}\right); \qquad
\overline{W}_{abc6}=-\dfrac{1}{2}t\left(\qq{ab}_{,c}+\dfrac{1}{4}\qq{bd}_{,d}\delta_{ac}-
\dfrac{1}{4}\qq{ad}_{,d}\delta_{bc}\right); \notag \\
\overline{W}_{ab56}&=\overline{R}_{ab56}=\qd{ab}-\dfrac{1}{4}t^2(\qt{ac}\qq{bc}-\qt{bc}\qq{ac});
\notag \qquad \quad
\overline{W}_{5ab6}=-\dfrac{1}{2}\qd{ab}-\dfrac{1}{4}t^2\qq{ac}\qt{bc}+
\dfrac{1}{16}t^2\sum_{c,d=1}^4\qq{cd}\qt{cd}\delta_{ab}; \notag\\
\overline{W}_{56a5}&=\dfrac{1}{8}\qq{ab}_{,b}; \quad
\overline{W}_{56a6}=-\dfrac{1}{8}\qt{ab}_{,b}; \qquad \qquad \qquad \qquad
\overline{W}_{5656}=\dfrac{3}{5t^2}+\dfrac{S}{20}-\dfrac{3}{40}t^2(|\qt{ab}|^2+|\qq{ab}|^2). \notag
\end{align*}
\normalsize
Let $(Z,g_t,J)$ be a K\"{a}hler-Einstein manifold. By \eqref{riemtwist},
the components of the Riemann tensor $\ol{\operatorname{Riem}}$ of $Z$ are
\small
\noindent
\begin{multicols}{2}
\noindent
\begin{align} \label{riemkahl}
\ol{R}_{5656}&=\dfrac{1}{t^2}, \qquad \ol{R}_{56a6}=\ol{R}_{565b}=0;\\
\ol{R}_{5ab5}&=\ol{R}_{6ab6}=-\dfrac{1}{4t^2}\del{ab}; \notag\\
\ol{R}_{5126}&=\ol{R}_{5346}=-\ol{R}_{5216}=-\ol{R}_{5436}=-\dfrac{1}{4t^2};
\notag\\
\ol{R}_{5ab6}&=0 \quad \mbox{ for } (a,b)\neq (1,2),(3,4); \notag\\
\ol{R}_{abc5}&=\ol{R}_{abc6}=0; \notag\\
\ol{R}_{1256}&=\ol{R}_{3456}=\dfrac{1}{2t^2}, \notag\\
\ol{R}_{ab56}&=0 \quad
\mbox{ for } (a,b)\neq (1,2),(3,4); \notag
\end{align}
\begin{align*}
\ol{R}_{1212}&=R_{1212}, \quad \ol{R}_{3434}=R_{3434};\notag\\
\ol{R}_{1313}&=R_{1313}-\dfrac{3}{4t^2}, \quad
\ol{R}_{4242}=R_{4242}-\dfrac{3}{4t^2};\notag\\
\ol{R}_{1414}&=R_{1414}-\dfrac{3}{4t^2}, \quad
\ol{R}_{2323}=R_{2323}-\dfrac{3}{4t^2};\notag\\
\ol{R}_{1234}&=R_{1234}+\dfrac{1}{2t^2};\notag\\
\ol{R}_{1342}&=R_{1342}-\dfrac{1}{4t^2};\notag\\
\ol{R}_{1423}&=R_{1423}-\dfrac{1}{4t^2};\notag\\
\ol{R}_{abcd}&=R_{abcd} \quad \mbox{ for the other components} . \notag
\end{align*}
\end{multicols}
\normalsize
The components of $\nabla\ol{\operatorname{Riem}}$ are
\vspace{0cm}
\small
\begin{align} \label{covriemkahl}
\ol{R}_{5656,t}&=\ol{R}_{565b,t}=\ol{R}_{56a6,t}=
\ol{R}_{5ab5,t}=\ol{R}_{6ab6,t}=\ol{R}_{5ab6,t}=\ol{R}_{ab56,t}=0; \\
\ol{R}_{abcd,e}&=R_{abcd,e}, \quad \ol{R}_{abcd,5}
=\ol{R}_{abcd,6}=
\ol{R}_{abc5,5}=\ol{R}_{abc5,6}=\ol{R}_{abc6,5}=\ol{R}_{abc6,6}=0;\notag
\end{align}
\begin{align*}
\begin{cases}
	2t\ol{R}_{12c5,1}=-R_{12c3};\\
	2t\ol{R}_{12c5,2}=R_{12c4}; \\
	2t\ol{R}_{12c5,3}=\frac{1}{t^2}\del{2c}-R_{121c}; \\
	2t\ol{R}_{12c5,4}=\frac{1}{t^2}\del{1c}-R_{12c2};
\end{cases}
\begin{cases}
	2t\ol{R}_{13c5,1}=\frac{1}{t^2}\del{1c}-R_{13c3}; \\
	2t\ol{R}_{13c5,2}=R_{13c4}; \\
	2t\ol{R}_{13c5,3}=\frac{1}{t^2}\del{3c}-R_{131c}; \\
	2t\ol{R}_{13c5,4}=-R_{13c2};
\end{cases}
\begin{cases}
	2t\ol{R}_{14c5,1}=-R_{14c3}; \\
	2t\ol{R}_{14c5,2}=R_{14c4}-\frac{1}{t^2}\del{1c}; \\
	2t\ol{R}_{14c5,3}=\frac{1}{t^2}\del{4c}-R_{141c}; \\
	2t\ol{R}_{14c5,4}=-R_{14c2};
\end{cases}
\end{align*}
\begin{align*}
\begin{cases}
	2t\ol{R}_{23c5,1}=\frac{1}{t^2}\del{2c}-R_{23c3}; \\
	2t\ol{R}_{23c5,2}=R_{23c4}; \\
	2t\ol{R}_{23c5,3}=R_{23c1}; \\
	2t\ol{R}_{23c5,4}=R_{232c}-\frac{1}{t^2}\del{3c};
\end{cases}
\begin{cases}
	2t\ol{R}_{24c5,1}=R_{42c3};\\
	2t\ol{R}_{24c5,2}=R_{424c}-\frac{1}{t^2}\del{2c};\\
	2t\ol{R}_{24c5,3}=-R_{42c1};\\
	2t\ol{R}_{24c5,4}=R_{42c2}-\frac{1}{t^2}\del{4c};
\end{cases}
\begin{cases}
	2t\ol{R}_{34c5,1}=R_{343c}-\frac{1}{t^2}\del{4c};\\
	2t\ol{R}_{34c5,2}=R_{34c4}-\frac{1}{t^2}\del{3c};\\
	2t\ol{R}_{34c5,3}=R_{34c1};\\
	2t\ol{R}_{34c5,4}=-R_{34c2}
\end{cases}
\end{align*}
\begin{align*}
\begin{cases}
	2t\ol{R}_{12c6,1}=-R_{12c4};\\
	2t\ol{R}_{12c6,2}=-R_{12c3}; \\
	2t\ol{R}_{12c6,3}=R_{12c2}-\frac{1}{t^2}\del{1c}; \\
	2t\ol{R}_{12c6,4}=\frac{1}{t^2}\del{2c}-R_{121c};
\end{cases}
\begin{cases}
	2t\ol{R}_{13c6,1}=-R_{13c4}; \\
	2t\ol{R}_{13c6,2}=\frac{1}{t^2}\del{1c}-R_{13c3}; \\
	2t\ol{R}_{13c6,3}=R_{13c2}; \\
	2t\ol{R}_{13c6,4}=\frac{1}{t^2}\del{3c}-R_{131c};
\end{cases}
\begin{cases}
	2t\ol{R}_{14c6,1}=\frac{1}{t^2}\del{1c}-R_{14c4}; \\
	2t\ol{R}_{14c6,2}=-R_{14c3}; \\
	2t\ol{R}_{14c6,3}=R_{14c2}; \\
	2t\ol{R}_{14c6,4}=\frac{1}{t^2}\del{4c}-R_{141c};
\end{cases}
\end{align*}
\begin{align*}
\begin{cases}
	2t\ol{R}_{23c6,1}=-R_{23c4}; \\
	2t\ol{R}_{23c6,2}=\frac{1}{t^2}\del{2c}-R_{23c3}; \\
	2t\ol{R}_{23c6,3}=\frac{1}{t^2}\del{3c}-R_{232c}; \\
	2t\ol{R}_{23c6,4}=-R_{231c};
\end{cases}
\begin{cases}
	2t\ol{R}_{24c6,1}=\frac{1}{t^2}\del{2c}-R_{424c};\\
	2t\ol{R}_{24c6,2}=R_{42c3};\\
	2t\ol{R}_{24c6,3}=\frac{1}{t^2}\del{4c}-R_{42c2};\\
	2t\ol{R}_{24c6,4}=R_{421c};
\end{cases}
\begin{cases}
	2t\ol{R}_{34c6,1}=\frac{1}{t^2}\del{3c}-R_{34c4};\\
	2t\ol{R}_{34c6,2}=R_{343c}-\frac{1}{t^2}\del{4c};\\
	2t\ol{R}_{34c6,3}=R_{34c2};\\
	2t\ol{R}_{34c6,4}=-R_{341c}
\end{cases}
\end{align*}
\normalsize
The general expressions of the components of $\nabla\ol{\ricc}$ are
\small
\begin{multicols}{2}
\begin{align} \label{nablaricc}
\ol{R}_{ab,c}&=R_{ab,c}-\dfrac{t^2}{2}[
\qt{ad}_{,c}\qt{bd}+\qt{bd}_{,c}\qt{ad}+\qq{ad}_{,c}\qq{bd}+
\qq{bd}_{,c}\qq{ad}]\\
&-\dfrac{t^2}{4}[\qt{ad}_{,d}\qt{bc}+\qt{bd}_{,d}\qt{ac}+
\qq{ad}_{,d}\qq{bc}+\qq{bd}_{,d}\qq{ac}];\notag\\
\ol{R}_{ab,5}&=\dfrac{t}{2}(\qd{ac}\qq{bc}+\qd{bc}\qq{ac})-
\dfrac{t}{2}(R_{ac}\qt{bc}+R_{bc}\qt{ac})+\notag\\
&+\dfrac{t^3}{4}[\qt{ac}(\qt{cd}\qt{bd}+\qq{cd}\qq{bd})+
\qt{bc}(\qt{cd}\qt{ad}+\qq{cd}\qq{ad})];\notag\\
\ol{R}_{ab,6}&=\dfrac{t}{2}(\qd{ac}\qt{bc}+\qd{bc}\qt{ac})-
\dfrac{t}{2}(R_{ac}\qq{bc}+R_{bc}\qq{ac})+\notag\\
&+\dfrac{t^3}{4}[\qq{ac}(\qt{cd}\qt{bd}+\qq{cd}\qq{bd})+
\qq{bc}(\qt{cd}\qt{ad}+\qq{cd}\qq{ad})];\notag\\
\ol{R}_{a5,b}&=\dfrac{t}{2}\qt{ad}_{,db}-
\dfrac{t}{2}\sq{\qt{ab}\pa{\dfrac{t^2}{4}\abs{\qt{cd}}^2+
		\dfrac{1}{t^2}}+\dfrac{t^2}{4}\qq{ab}\qt{dc}\qq{dc}}+\notag\\
&+\dfrac{t}{2}\qt{db}\sq{R_{ad}-\dfrac{t^2}{2}(\qt{ac}\qt{dc}+
	\qq{ac}\qq{dc})};\notag\\
\ol{R}_{a6,b}&=\dfrac{t}{2}\qq{ad}_{,db}-
\dfrac{t}{2}\sq{\qq{ab}\pa{\dfrac{t^2}{4}\abs{\qq{cd}}^2+
		\dfrac{1}{t^2}}+\dfrac{t^2}{4}\qt{ab}\qt{dc}\qq{dc}}+\notag\\
&+\dfrac{t}{2}\qq{db}\sq{R_{ad}-\dfrac{t^2}{2}(\qt{ac}\qt{dc}+
	\qq{ac}\qq{dc})};\notag
\end{align}
\begin{align*}
	\ol{R}_{a5,5}&=-\dfrac{t^2}{4}\qt{cd}_{,d}\qt{ac};\\
	\ol{R}_{a5,6}&=\dfrac{1}{2}\qd{ac}_{,c}-\dfrac{t^2}{4}
	\qt{cd}_{,d}\qq{ac};\notag\\
\ol{R}_{a6,5}&=-\dfrac{1}{2}\qd{ac}_{,c}
-\dfrac{t^2}{4}\qq{cd}_{,d}\qt{ac}; \\
\ol{R}_{a6,6}&=-\dfrac{t^2}{4}
\qq{cd}_{,d}\qq{ac};\notag\\
\ol{R}_{55,a}&=\dfrac{t^2}{2}(\qt{bc}\qt{bc}_{,a}+\qt{bc}_{,c}\qt{ba});\\
\ol{R}_{55,5}&=0, \qquad \ol{R}_{55,6}=\dfrac{t}{2}\qd{dc}\qt{dc};
\notag\\
\ol{R}_{66,a}&=\dfrac{t^2}{2}(\qq{bc}\qq{bc}_{,a}+\qq{bc}_{,c}\qq{ba});
\\
\ol{R}_{66,5}&=-\dfrac{t}{2}\qd{dc}\qq{dc}, \qquad \ol{R}_{66,6}=0;
\notag\\
\ol{R}_{56,a}&=\dfrac{t^2}{4}\left[\qt{bc}\qq{bc}_{,a}+\qq{bc}\qt{bc}_{,a}+
\right. \\
&+\left.
\qt{ba}\qq{bc}_{,c}+\qq{ba}\qt{bc}_{,c}\right];\notag\\
\ol{R}_{56,5}&=-\dfrac{t}{4}\qd{ab}\qt{ab}, \qquad
\ol{R}_{56,6}=\dfrac{t}{4}\qd{ab}\qq{ab}. \notag
\end{align*}
\end{multicols}
\normalsize
The covariant derivative of the scalar curvature $\ol{S}$ has components
\begin{align} \label{nablascal}
\ol{S}_a&=S_a-\dfrac{t^2}{2}[\qt{bc}\qt{bc}_{,a}+\qq{bc}\qq{bc}_{,a}];
\\
\ol{S}_5&=\dfrac{t}{2}\qd{bc}\qq{bc}, \qquad
\ol{S}_6=-\dfrac{t}{2}\qd{bc}\qt{bc}. \notag
\end{align}

\section{Covariant derivative of the almost complex structures and differential of the K\"ahler forms}\label{appb}

In this appendix we list all the components of the covariant derivative of the almost complex structures $J^{\pm}$ on $Z$.  Recall that

\begin{align}\label{eq:defJpmsixdim}
J^{\pm} &= \sum_{k=1}^3\pa{\theta^{2k-1}\otimes e_{2k}-\theta^{2k}\otimes e_{2k-1}}\\ \nonumber &=\theta^1\otimes e_2 - \theta^2\otimes e_1 + \theta^3\otimes e_4 - \theta^4\otimes e_3 \pm \theta^5\otimes e_6 \mp \theta^6\otimes e_5;
\end{align}
using the same notation of the article, we write $J^+=J$, $J^-=\mathbf{J}$.

\

\textbf{Computation of} $\nabla J$ \textbf{on} $(Z,g_t,J)$:

The covariant derivative $\nabla J$ of an almost complex
structure is defined as
\begin{equation} \label{nablajcomp}
	\nabla J=J_{q,t}^p\theta^t\otimes\theta^q\otimes e_p, \mbox{ where }
	J_{q,t}^p=dJ_q^p-J_r^p\omega_q^r+J_q^r\omega_r^p, \mbox { }
	J_{p,t}^q=-J_{q,t}^p,
\end{equation}
with respect to a local orthonormal frame $\{\theta_p\}$ and its dual
frame $\{e_q\}$.
Using \eqref{nablajcomp}, a long but straightforward computation shows that:

\noindent
\begin{multicols}{2}
\noindent
\begin{align} \label{nablajplus}
	J^1_{2, t} &= J^3_{4, t} = J^5_{6, t} =0; \\
	J^1_{3, a} &= 0; \notag\\
	J^1_{3, 5} &= -\frac{1}{2}t\pa{\qt{14}+\qt{23}}; \notag\\
	J^1_{3, 6} &= \frac{1}{2}t\sq{\frac{2}{t^2}-\pa{\qq{14}+\qq{23}}}; \notag\\
	J^1_{4, a} &=0; \notag\\
	J^1_{4, 5} &= -\frac{1}{2}t\sq{\frac{2}{t^2}-\pa{\qt{13}+\qt{42}}};
	\notag\\
	J^1_{4, 6} &= \frac{1}{2}t\pa{\qq{13}+\qq{42}}; \notag\\
	J^1_{5, a} &= -\frac{1}{2}t\pa{\qt{2a}+\qq{1a}}; \notag\\
	J^1_{5, 5} &= J^1_{5, 6}= 0; \notag\\
	J^1_{6, a} &= \frac{1}{2}t\pa{\qt{1a}-\qq{2a}}; \notag
\end{align}
\begin{align*}
	J^1_{6, 5} &= J^1_{6, 6}= 0; \notag\\
	J^2_{3, t} &=J^1_{4, t}; \notag\\
	J^2_{4, t} &= -J^1_{3, t} \notag\\
	J^2_{5, t} &= J^1_{6, t}\notag\\
	J^2_{6, t} &= -J^1_{5, t} \notag\\
	J^3_{5, a} &= -\frac{1}{2}t\pa{\qt{4a}+\qq{3a}}; \notag\\
	J^3_{5, 5} &= J^3_{5, 6}= 0; \notag\\
	J^3_{6, a} &= \frac{1}{2}t\pa{\qt{3a}-\qq{4a}}; \notag\\
	J^3_{6, 5} &= J^3_{6, 6}= 0; \notag\\
	J^4_{5, t} &= J^3_{6, t}\notag \\
	J^4_{6, t} &= -J^3_{5, t}.\notag
\end{align*}
\end{multicols}
The square norm $|\nabla J|^2=\sum_{p,q,t=1}^6J_{p,q}^tJ_{p,q}^t$ is given by
\begin{align} \label{squarejplus}
\dfrac{1}{t^2}|\nabla J|^2&=[(\qt{14}+\qt{23})^2+(\qq{14}+\qq{23})^2+
(\qt{13}+\qt{42})^2+(\qq{13}+\qq{42})^2+\\
&+(\qt{13}-\qq{14})^2+(\qt{42}-\qq{14})^2+(\qt{13}-\qq{23})^2+(\qt{14}+\qq{42})^2+\notag\\
&+(\qt{14}+\qq{13})^2+(\qq{23}-\qt{42})^2+(\qq{14}-\qt{13})^2+(\qt{23}+\qq{42})^2]+\notag\\
&+2[\qt{12}^2+\qq{12}^2+\qt{34}^2+\qq{34}^2]-\dfrac{4}{t^2}[(\qt{13}+\qt{42})+(\qq{14}+\qq{23})]+
\dfrac{8}{t^4}.\notag
\end{align}
\noindent
\textbf{Computation} \textbf{of} $\nabla \mathbf{J}$ \textbf{on} $(Z,g_t,J)$:

Again, using  \eqref{nablajcomp}, we have:
\noindent
\begin{multicols}{2}
	\noindent
	\begin{align} \label{nablajmin}
		\mathbf{J}^1_{2, t} &= \mathbf{J}^3_{4, t} = \mathbf{J}^5_{6, t} =0; \\ \mathbf{J}^1_{3, a} &= 0; \notag\\
		\mathbf{J}^1_{3, 5} &= -\frac{1}{2}t\pa{\qt{14}+\qt{23}}; \notag\\
		\mathbf{J}^1_{3, 6} &= \frac{1}{2}t\sq{\frac{2}{t^2}-\pa{\qq{14}+\qq{23}}}; \notag\\
		\mathbf{J}^1_{4, a} &=0; \notag\\ 
		\mathbf{J}^1_{4, 5} &= -\frac{1}{2}t\sq{\frac{2}{t^2}-\pa{\qt{13}+\qt{42}}}; \notag
		 \end{align}
	 \begin{align*}
		 \mathbf{J}^1_{4, 6} &= \frac{1}{2}t\pa{\qq{13}+\qq{42}}; \notag\\ \mathbf{J}^1_{5, a} &= -\frac{1}{2}t\pa{\qt{2a}-\qq{1a}};\notag\\
		\mathbf{J}^1_{5, 5} &= \mathbf{J}^1_{5, 6}= 0; \notag\\
		\mathbf{J}^1_{6, a} &= -\frac{1}{2}t\pa{\qt{1a}+\qq{2a}};\notag\\
		\mathbf{J}^1_{6, 5} &= \mathbf{J}^1_{6, 6}= 0; \notag
	\end{align*}
	\begin{align*}
		\mathbf{J}^2_{3, t} &=\mathbf{J}^1_{4, t}; \notag\\
		\mathbf{J}^2_{4, t} &= -\mathbf{J}^1_{3, t} \notag\\ 
		\mathbf{J}^2_{5, t} &= -\mathbf{J}^1_{6, t} \notag\\ 
		\mathbf{J}^2_{6, t} &= \mathbf{J}^1_{5, t} \notag\\ 
		\mathbf{J}^3_{5, a} &= -\frac{1}{2}t\pa{\qt{4a}-\qq{3a}}; \notag
	\end{align*}
	\begin{align*}
		\mathbf{J}^3_{5, 5} &= \mathbf{J}^3_{5, 6}= 0; \notag\\ \mathbf{J}^3_{6, a} &= -\frac{1}{2}t\pa{\qt{3a}+\qq{4a}}; \notag\\ \mathbf{J}^3_{6, 5} &= \mathbf{J}^3_{6, 6}= 0;\notag \\ \mathbf{J}^4_{5, t} &= -\mathbf{J}^3_{6, t} \notag\\ \mathbf{J}^4_{6, t} &= \mathbf{J}^3_{5, t}.\notag
	\end{align*}
\end{multicols}
\

\textbf{K\"{a}hler forms of} $J$ \textbf{and} $\JJ$:

denoting by $\omega_+$ and $\omega_-$ the K\"{a}hler forms of $J$ and $\JJ$, respectively, we have:

\begin{align} \label{kahtwistplus}
d\omega_+&=-t\qq{12}\theta^1\wedge\theta^2\wedge\theta^5+t\qt{12}\theta^1\wedge\theta^2\wedge\theta^6-
t\qq{13}\theta^1\wedge\theta^3\wedge\theta^5+\\
&+\left(t\qt{13}-\dfrac{1}{t}\right)\theta^1\wedge\theta^3\wedge\theta^6
+\left(\dfrac{1}{t}-t\qq{14}\right)\theta^1\wedge\theta^4\wedge\theta^5+ \notag\\
&+t\qt{14}\theta^1\wedge\theta^4\wedge\theta^6+\left(\dfrac{1}{t}-t\qq{23}\right)\theta^2\wedge\theta^3\wedge\theta^5
+t\qt{23}\theta^2\wedge\theta^3\wedge\theta^6+\notag\\
&-t\qq{42}\theta^4\wedge\theta^2\wedge\theta^5+\left(t\qt{42}-\dfrac{1}{t}\right)\theta^4\wedge\theta^2\wedge\theta^6
-t\qq{34}\theta^3\wedge\theta^4\wedge\theta^5+\notag\\
&+t\qt{34}\theta^3\wedge\theta^4\wedge\theta^6 \notag.
\end{align}

\begin{align} \label{kahtwistmin}
d\omega_-&=t\qq{12}\theta^1\wedge\theta^2\wedge\theta^5-t\qt{12}\theta^1\wedge\theta^2\wedge\theta^6+
t\qq{13}\theta^1\wedge\theta^3\wedge\theta^5+\\
&-\left(t\qt{13}+\dfrac{1}{t}\right)\theta^1\wedge\theta^3\wedge\theta^6
+\left(\dfrac{1}{t}+t\qq{14}\right)\theta^1\wedge\theta^4\wedge\theta^5+\notag\\
&-t\qt{14}\theta^1\wedge\theta^4\wedge\theta^6+\left(\dfrac{1}{t}+t\qq{23}\right)\theta^2\wedge\theta^3\wedge\theta^5
-t\qt{23}\theta^2\wedge\theta^3\wedge\theta^6+\notag\\
&+t\qq{42}\theta^4\wedge\theta^2\wedge\theta^5-\left(t\qt{42}+\dfrac{1}{t}\right)\theta^4\wedge\theta^2\wedge\theta^6
+t\qq{34}\theta^3\wedge\theta^4\wedge\theta^5+\notag\\
&-t\qt{34}\theta^3\wedge\theta^4\wedge\theta^6. \notag
\end{align}

As far as the codifferentials of $\omega_+$ and $\omega_-$ are concerned, we have:
\begin{equation} \label{codifftwist}
\delta\omega_+=\delta\omega_-=t(\qt{12}+\qt{34})\theta^5+t(\qq{12}+\qq{34})\theta^6.
\end{equation}

\section{Hessian and Laplacian of $J$ and $\textbf{J}$} \label{appendc}
By definition, we have
\begin{equation*}
\nabla^2J = J^p_{q, rt}\theta^t\otimes\theta^r\otimes\theta^q\otimes e_p,
\end{equation*}
and
\begin{equation*}
J^p_{q, rt}\theta^t = dJ^p_{q, r}-J^p_{t, r}\theta^t_q-J^p_{q, t}\theta^t_r+J^t_{q, r}\theta^p_t; \quad J^p_{q, rt}= - J^q_{p, rt}.
\end{equation*}
\footnotesize
\begin{align*}
J^1_{2, rt}\theta^t &= -\frac{1}{2}t\sq{J^1_{5, r}\pa{\qt{2a}+\qq{1a}}+J^1_{6, r}\pa{\qq{2a}-\qt{1a}}}\theta^a \\
&-\frac{1}{2}t\set{J^1_{3, r}\pa{\qt{14}+\qt{23}} + J^1_{4, r}\sq{\frac{2}{t^2}-\pa{\qt{13}+\qt{42}}}}\theta^5 \\ &+\frac{1}{2}t\set{J^1_{3, r}\sq{\frac{2}{t^2}-\pa{\qq{14}+\qq{23}}} + J^1_{4, r}\pa{\qq{13}+\qq{42}}}\theta^6;\\
J^3_{4, rt}\theta^t &= \frac{1}{2}t\sq{J^3_{6, r}\pa{\qt{3a}-\qq{4a}}-J^3_{5, r}\pa{\qt{4a}+\qq{3a}}}\theta^a \\
&-\frac{1}{2}t\set{J^1_{3, r}\pa{\qt{14}+\qt{23}} + J^1_{4, r}\sq{\frac{2}{t^2}-\pa{\qt{13}+\qt{42}}}}\theta^5 \\ &+\frac{1}{2}t\set{J^1_{3, r}\sq{\frac{2}{t^2}-\pa{\qq{14}+\qq{23}}} + J^1_{4, r}\pa{\qq{13}+\qq{42}}}\theta^6; \\
J^5_{6, rt}\theta^t &= \frac{1}{2}t\sq{J^1_{6, r}\pa{\qt{1a}-\qq{2a}}+J^2_{6, r}\pa{\qt{2a}+\qq{1a}}+J^3_{6, r}\pa{\qt{3a}-\qq{4a}}+J^4_{6, r}\pa{\qt{4a}+\qq{3a}}}\theta^a;\\
J^1_{3, at}\theta^t &= \frac{1}{4}t^2\set{\pa{\qt{3a}-\qq{4a}}\qq{1b}-\pa{\qt{4a}+\qq{3a}}\qt{1b}+\pa{\qt{2a}+\qq{1a}}\qt{3b}}\theta^b \\ & +\frac{1}{4}t^2\set{-\pa{\qt{1a}-\qq{2a}}\qq{3b}+\pa{\qt{14}+\qt{23}}\qt{ab}-\sq{\frac{2}{t^2}-\pa{\qq{14}+\qq{23}}}\qq{ab}}\theta^b;\\
J^1_{3, 5t}\theta^t &= -\frac{1}{2}t\pa{\qt{14}+\qt{23}}_{a}\theta^a +\frac{1}{4}t^2\pa{\qt{12}+\qt{34}}\sq{\frac{4}{t^2}-\pa{\qt{q3}+\qt{42}}}\theta^5 \\ &-\frac{1}{4}t^2\sq{\pa{\qq{12}+\qq{34}}\pa{\qt{13}+\qt{42}}}\theta^6; \\
J^1_{3, 6t}\theta^t &= -\frac{1}{2}t\pa{\qq{14}+\qq{23}}_{a}\theta^a +\frac{1}{4}t^2\sq{\frac{4}{t^2}\pa{\qq{12}+\qq{34}}-\pa{\qq{13}+\qq{42}}\pa{\qt{12}+\qt{34}}}\theta^5 \\&-\frac{1}{4}t^2\pa{\qq{13}+\qq{42}}\pa{\qq{12}+\qq{34}}\theta^6;\\
J^1_{4, at}\theta^t &= \frac{1}{4}t^2\set{\pa{\qt{3a}-\qq{4a}}\qt{1b}+\pa{\qt{4a}+\qq{3a}}\qq{1b}+\pa{\qt{2a}+\qq{1a}}\qt{4b}}\theta^b \\ &= +\frac{1}{4}t^2\set{-\pa{\qt{1a}-\qq{2a}}\qq{4b}-\pa{\qq{13}+\qq{42}}\qq{ab}+\sq{\frac{2}{t^2}-\pa{\qt{13}+\qt{42}}}\qt{ab}}\theta^b;\\
J^1_{4, 5t}\theta^t &= \frac{1}{2}t\pa{\qt{13}+\qt{42}}_{a}\theta^a -\frac{1}{4}t^2\pa{\qt{12}+\qt{34}}\pa{\qt{14}+\qt{23}}\theta^5 \\ &+\frac{1}{4}t^2\sq{\frac{4}{t^2}\pa{\qt{12}+\qt{34}}-\pa{\qt{14}+\qt{23}}\pa{\qq{12}+\qq{34}}}\theta^6;\\
J^1_{4, 6t}\theta^t &= \frac{1}{2}t\pa{\qq{13}+\qq{42}}_{a}\theta^a -\frac{1}{4}t^2\pa{\qt{12}+\qt{34}}\pa{\qq{14}+\qq{23}}\theta^5 \\&+\frac{1}{4}t^2\pa{\qq{12}+\qq{34}}\sq{\frac{4}{t^2}-\pa{\qq{14}+\qq{23}}}\theta^6;\\
J^1_{5, at}\theta^t &= -\frac{1}{2}t\pa{\qt{2a}+\qq{1a}}_b\theta^b + \frac{1}{2}\pa{\qd{1a}-\qt{4a}}\theta^5 + \frac{1}{2}\pa{\qt{3a}-\qd{2a}}\theta^6 \\ &+\frac{1}{4}t^2\sq{\pa{\qt{2c}+\qq{1c}}\qt{ac}-\pa{\qt{1a}-\qq{2a}}\qt{12}+\pa{\qt{4a}+\qq{3a}}\qt{13}-\pa{\qt{3a}-\qq{4a}}\qt{14}}\theta^5 \\
&+\frac{1}{4}t^2\sq{\pa{\qt{2c}+\qq{1c}}\qt{ac}-\pa{\qt{1a}-\qq{2a}}\qq{12}+\pa{\qt{4a}+\qq{3a}}\qq{13}-\pa{\qt{3a}-\qq{4a}}\qq{14}}\theta^6; \\
J^1_{5, 5t}\theta^t &= \frac{1}{2}t\pa{J^1_{c, 5}+J^1_{5, c}}\qt{cb}\theta^b; \qquad \qquad
J^1_{5, 6t}\theta^t = \frac{1}{2}t\pa{J^1_{c, 6}\qt{cb}+J^1_{5, c}\qq{cb}}\theta^b;
\end{align*}
\begin{align*}
J^1_{6, at}\theta^t &= \frac{1}{2}t\pa{\qt{1a}-\qq{2a}}_b\theta^b + \frac{1}{2}\pa{\qd{2a}-\qq{4a}}\theta^5 + \frac{1}{2}\pa{\qd{1a}+\qq{3a}}\theta^6 \\ &-\frac{1}{4}t^2\sq{\pa{\qt{1c}-\qq{2c}}\qt{ac}+\pa{\qt{2a}+\qq{1a}}\qt{12}+\pa{\qt{3a}-\qq{4a}}\qt{13}+\pa{\qt{4a}+\qq{3a}}\qt{14}}\theta^5 \\
&-\frac{1}{4}t^2\sq{\pa{\qt{1c}-\qq{2c}}\qq{ac}+\pa{\qt{3a}+\qq{1a}}\qq{12}+\pa{\qt{3a}-\qq{4a}}\qq{13}+\pa{\qt{4a}+\qq{3a}}\qq{14}}\theta^6; \\
J^1_{6, 5t}\theta^t &= \frac{1}{2}t\pa{J^1_{6, c}\qt{cb}+J^1_{c, 5}\qq{cb}}\theta^b; \\
J^1_{6, 6t}\theta^t &= \frac{1}{2}t\pa{J^1_{c, 6}+J^1_{6, c}}\qq{cb}\theta^b;\\
J^3_{5, at}\theta^t &= -\frac{1}{2}t\pa{\qt{4a}+\qq{3a}}_b\theta^b + \frac{1}{2}\pa{\qd{3a}+\qt{2a}}\theta^5 - \frac{1}{2}\pa{\qd{4a}+\qt{1a}}\theta^6 \\ &+\frac{1}{4}t^2\sq{\pa{\qt{4c}+\qq{3c}}\qt{ac}-\pa{\qt{2a}+\qq{1a}}\qt{13}+\pa{\qt{1a}-\qq{2a}}\qt{23}-\pa{\qt{3a}-\qq{4a}}\qt{34}}\theta^5 \\
&+\frac{1}{4}t^2\sq{\pa{\qt{4c}+\qq{3c}}\qq{ac}-\pa{\qt{2a}+\qq{1a}}\qq{13}+\pa{\qt{1a}-\qq{2a}}\qq{23}-\pa{\qt{3a}-\qq{4a}}\qq{34}}\theta^6; \\
J^3_{5, 5t}\theta^t &= \frac{1}{2}t\pa{J^3_{c, 5}+J^3_{5, c}}\qt{cb}\theta^b; \\
J^3_{5, 6t}\theta^t &= \frac{1}{2}t\pa{J^3_{c, 6}\qt{cb}+J^3_{5, c}\qq{cb}}\theta^b;\\
J^3_{6, at}\theta^t &= \frac{1}{2}t\pa{\qt{3a}-\qq{4a}}_b\theta^b + \frac{1}{2}\pa{\qd{4a}+\qq{2a}}\theta^5 + \frac{1}{2}\pa{\qd{3a}-\qq{1a}}\theta^6 \\ &-\frac{1}{4}t^2\sq{\pa{\qt{3c}-\qq{4c}}\qt{ac}-\pa{\qt{1a}-\qq{2a}}\qt{13}-\pa{\qt{2a}+\qq{1a}}\qt{23}+\pa{\qt{4a}+\qq{3a}}\qt{34}}\theta^5 \\
&-\frac{1}{4}t^2\sq{\pa{\qt{3c}-\qq{4c}}\qq{ac}-\pa{\qt{1a}-\qq{2a}}\qq{13}-\pa{\qt{2a}+\qq{1a}}\qq{23}+\pa{\qt{4a}+\qq{3a}}\qq{34}}\theta^6; \\
J^3_{6, 5t}\theta^t &= \frac{1}{2}t\pa{J^3_{6, c}\qt{cb}+J^3_{c, 5}\qq{cb}}\theta^b; \\
J^3_{6, 6t}\theta^t &= \frac{1}{2}t\pa{J^3_{c, 6}+J^3_{6, c}}\qq{cb}\theta^b;\\
J^4_{5, rt} &= J^3_{6, rt}; \qquad
J^4_{6, rt} = -J^3_{5, rt};\qquad
J^2_{3, rt} = J^1_{4, rt}; \qquad
J^2_{4, rt} = -J^1_{3, rt};\qquad
J^2_{5, rt} = J^1_{6, rt};\qquad
J^2_{6, rt} = -J^1_{5, rt}.
\end{align*}
\normalsize
The local components of the Laplacian $\Delta_JJ$ are
\begin{equation} \label{lapltwist}
\Delta_JJ^u_v := \pa{\Delta J-J\nabla_pJ\nabla_pJ}^u_v = \Delta J^u_v - J^u_tJ^t_{q, p}J^q_{v, p}.
\end{equation}
Explicitly, we obtain
\footnotesize
\begin{multicols}{2}
\begin{align*}
\Delta_JJ^1_2 &= \Delta_JJ^3_4 = \Delta_JJ^5_6 \equiv 0; \\ \Delta_JJ^1_3 &= \pa{\qt{12}+\qt{34}} + \frac{1}{4}t\sq{N^5_{14}\pa{\qq{12}+\qq{34}}-N^5_{13}\pa{\qt{12}+\qt{34}}} = - \Delta_JJ^2_4 \\ \Delta_JJ^1_4 &= \pa{\qq{12}+\qq{34}} + \frac{1}{4}t\sq{N^5_{14}\pa{\qt{12}+\qt{34}}-N^5_{13}\pa{\qq{12}+\qq{34}}} =  \Delta_JJ^2_3\\ \Delta_JJ^1_5 &= -\frac{1}{2}t\sum_{a=1}^4\pa{\qt{2a}+\qq{1a}}_a = -\Delta_JJ^2_6;
\end{align*}
\begin{align*}
 \Delta_JJ^1_6 &= \frac{1}{2}t\sum_{a=1}^4\pa{\qt{1a}-\qq{2a}}_a = \Delta_JJ^2_5;\\ \Delta_JJ^3_5 &= -\frac{1}{2}t\sum_{a=1}^4\pa{\qt{4a}+\qq{3a}}_a = -\Delta_JJ^4_6;\\ \Delta_JJ^3_6 &= \frac{1}{2}t\sum_{a=1}^4\pa{\qt{3a}-\qq{4a}}_a = \Delta_JJ^4_5.
\end{align*}
\end{multicols}
\normalsize
Note that
\footnotesize
\begin{align} \label{laplweyldiv}
\Delta_JJ^1_5 &= 2t\sq{\pa{\delta W^+}_{213}-\pa{\delta W^+}_{141}};\\
\Delta_JJ^1_6 &= 2t\sq{\pa{\delta W^+}_{214}+\pa{\delta W^+}_{131}};\notag\\
\Delta_JJ^3_5 &= 2t\sq{\pa{\delta W^+}_{413}+\pa{\delta W^+}_{314}};\notag\\
\Delta_JJ^3_6 &= 2t\sq{\pa{\delta W^+}_{414}-\pa{\delta W^+}_{313}}.\notag
\end{align}
\normalsize

\section{Computation of $N_J$, $N_{\textbf{J}}$} \label{appc}

Here we consider the Nijenhuis tensors $N_J$ and $N_{\JJ}$ of
$J=J^+$ and $\JJ=J^-$, respectively.
We have
\begin{align*}
N_J &= N^p_{tq}\theta^t\otimes\theta^q\otimes e_p, \quad N^p_{tq}=-N^p_{qt},\\
N_{\JJ} &= \mathbf{N}^p_{tq}\theta^t\otimes\theta^q\otimes e_p, \quad \mathbf{N}^p_{tq}=-\mathbf{N}^p_{qt},
\end{align*}
where $N^p_{tq}=J_t^rJ_{r,q}^p-J_q^rJ_{r,t}^p+J_q^sJ_{t,s}^p-J_t^sJ_{q,s}^p$
and analogously for $\mathbf{N}^p_{tq}$.

Using the definition, \eqref{nablajplus} and \eqref{nablajmin},
we have
\footnotesize
\noindent
\begin{multicols}{2}
\noindent
\begin{align} \label{nijcompplus}
N^a_{pq} &=0=N^5_{12}=N^5_{34}=N^6_{12}=N^6_{34}; \\
N^5_{13} &= -t({\qt{13}+\qt{42}-\qq{14}-\qq{23}}) = 2t({A_{33}-A_{22}}); \notag\\
N^5_{14} &=  -t({\qt{14}+\qt{23}+\qq{13}+\qq{42}}) \notag \\
&= -2t({\qt{14}+\qt{23}}) = -2t({\qq{13}+\qq{42}}) = -4tA_{23}; \notag
\end{align}
\begin{align*}
N^5_{13} &= -N^5_{24} = - N^6_{14} = -N^6_{23}; \notag\\
N^5_{14} &= N^5_{23} =  N^6_{13} = -N^6_{24}; \notag
\end{align*}
\end{multicols}
\noindent
\normalsize
Components of $\nabla N_J$: by
\begin{equation}
\nabla N_J = N^p_{tq, s}\theta^s\otimes\theta^t\otimes\theta^q\otimes e_p, \quad N^p_{tq, s}=-N^p_{qt, s}:
\end{equation}
\footnotesize
\begin{align} \label{nablanijplus}
N^a_{pq, 5} &= N^a_{pq, 6}=0; \\ N^a_{pq, b} &= -\frac{1}{2}\pa{N^5_{pq}\qt{ab}+N^6_{pq}\qq{ab}}. \notag
\end{align}
\noindent
\begin{multicols}{2}
\noindent
\begin{align*}
	N^5_{12, a} &=0; \\
	N^5_{12, 5} &= -N^5_{14}\sq{\frac{1}{t}-\frac{1}{2}t\pa{\qt{13}+\qt{42}}} + \frac{1}{4} N^5_{13}N^5_{14}; \\
	N^5_{12, 6} &= N^5_{13}\sq{\frac{1}{t}-\frac{1}{2}t\pa{\qq{14}+\qq{23}}}
	- \frac{1}{4} \pa{N^5_{14}}^2;\\
	N^5_{13, a} &= -t\sq{\pa{\qt{13}+\qt{42}}_a -
	\pa{\qq{14}+\qq{23}}_a};\\
	N^5_{13, 5} &=
	-\frac{1}{2}tN^5_{14}\pa{\qt{12}+\qt{34}}-2\pa{\qq{12}+\qq{34}};\\
	N^5_{13, 6} &= -\frac{1}{2}tN^5_{14}\pa{\qq{12}+\qq{34}}-2\pa{\qt{12}+\qt{34}};\\
	N^5_{14, a} &= -2t\pa{\qt{14}+\qt{23}}_a = -2t\pa{\qq{13}+\qq{42}}_a;\\
	N^5_{14, 5} &= \frac{1}{2}tN^5_{13}\pa{\qt{12}+\qt{34}}+2\pa{\qt{12}+\qt{34}};\\
	N^5_{14, 6} &=
	\frac{1}{2}tN^5_{13}\pa{\qq{12}+\qq{34}}-2\pa{\qq{12}+\qq{34}};\\
	N^5_{15, a} &= \frac{1}{2}t\pa{N^5_{13}\qt{3a}+N^5_{14}\qt{4a}};\\
	N^5_{15, 5} &= N^5_{15, 6}=0;\\
	N^5_{16, a} &= \frac{1}{2}t\pa{N^5_{13}\qq{3a}+N^5_{14}\qq{4a}};\\
	N^5_{16, 5} &= N^5_{16, 6}=0;\\
	N^5_{23, s} &= N^5_{14, s}; \\
	N^5_{24, s} &= -N^5_{13, s};\\
	N^5_{25, a} &= \frac{1}{2}t\pa{N^5_{14}\qt{3a}-N^5_{13}\qt{4a}};\\
	N^5_{25, 5} &= N^5_{25, 6}=0;\\
	N^5_{26, a} &= \frac{1}{2}t\pa{N^5_{14}\qq{3a}-N^5_{13}\qq{4a}}; \\ N^5_{26, 5} &= N^5_{26, 6}=0 ;
\end{align*}
\begin{align*}
	N^5_{34, s} &=  N^5_{12, s}; \\
	N^5_{35, a} &= -\frac{1}{2}t\pa{N^5_{13}\qt{1a}+N^5_{14}\qt{2a}}; \\ N^5_{35, 5} &= N^5_{35, 6}=0;\\
	N^5_{36, a} &= -\frac{1}{2}t\pa{N^5_{13}\qq{1a}+N^5_{14}\qq{2a}}; \\ N^5_{36, 5} &= N^5_{36, 6}=0; \\
	N^5_{45, a} &= -\frac{1}{2}t\pa{N^5_{14}\qt{1a}-N^5_{13}\qt{2a}}; \\ N^5_{45, 5} &= N^5_{45, 6}=0;\\
	N^5_{46, a} &= -\frac{1}{2}t\pa{N^5_{14}\qq{1a}-N^5_{13}\qq{2a}}; \\ N^5_{46, 5} &= N^5_{46, 6}=0;\\
	N^5_{56, s} &= 0;\\
	N^6_{12, a} &=0; \\
	N^6_{12, 5} &= N^5_{13}\sq{\frac{1}{t}-\frac{1}{2}t\pa{\qt{13}+\qt{42}}} + \frac{1}{4} \pa{N^5_{14}}^2; \\
	N^6_{12, 6} &= N^5_{14}\sq{\frac{1}{t}-\frac{1}{2}t\pa{\qq{14}+\qq{23}}} + \frac{1}{4} N^5_{13}N^5_{14};\\
	N^6_{13, s} &= N^5_{14, s};\\
	N^6_{14, s} &= -N^5_{13, s};\\
	N^6_{15, s} &= N^5_{25, s};\\
	N^6_{16, s} &= N^5_{26, s};\\
	N^6_{23, s} &= -N^5_{13, s};\\
	N^6_{24, s} &= -N^5_{14, s};\\
	N^6_{25, s} &= -N^5_{15, s};
\end{align*}
\begin{align*}
	N^6_{26, s} &= -N^5_{16, s};\\
	N^6_{34, s} &= N^6_{12, s}; \\
	N^6_{35, s} &= N^5_{46, s};\\
	N^6_{36, s} &= N^5_{46, s};
\end{align*}
\begin{align*}
	N^6_{45, s} &= -N^5_{35, s};\\
	N^6_{46, s} &= -N^5_{36, s};\\
	N^6_{56, s} &= 0.
\end{align*}
\end{multicols}
\normalsize
\noindent
Thus, we can now compute the components of the two divergences
\[
\diver N_J=N_{pq,t}^t\theta^p\otimes\theta^q, \qquad
\ol{\diver} N_J=N_{pt,t}^r\theta^p\otimes e_r.
\]
For the sake of simplicity, let
\[
\Gamma:=\qt{13}+\qt{42}-\qq{14}-\qq{23}=2(A_{22}+A_{33}).
\]
Then, we have
\footnotesize
\begin{align} \label{divnijplus1}
N_{12,t}^t&=N_{34,t}^t=0;\\
N_{13,t}^t&=-N_{24,t}^t=\dfrac{t}{2}\sq{(\qq{12}+\qq{34})\pa{N_{13}^5-
		\dfrac{8}{t}}-N_{14}^5(\qt{12}+\qt{34})}; \notag\\
N_{14,t}^t&=N_{23,t}^t=\dfrac{t}{2}\sq{(\qt{12}+\qt{34})\pa{N_{13}^5+
		\dfrac{8}{t}}+N_{14}^5(\qq{12}+\qq{34})}; \notag\\
N_{a5,t}^t&=N_{a6,t}^t=0. \notag
\end{align}
\noindent
\begin{multicols}{2}
\noindent
\begin{align} \label{divnijplus2}
N_{1t,t}^1&=-\dfrac{1}{2}\sq{N_{13}^5(\qt{13}-\qq{14})+N_{14}^5(\qt{14}
	+\qq{13})};\\
N_{2t,t}^1&=-\dfrac{1}{2}\sq{N_{14}^5(\qt{13}-\qq{14})-N_{13}^5(\qt{14}
	+\qq{13})}; \notag\\
N_{3t,t}^1&=-N_{4t,t}^2=\dfrac{1}{2}[N_{14}^5\qt{12}-N_{13}^5\qq{12}];
\notag\\
N_{4t,t}^1&=N_{3t,t}^2=-\dfrac{1}{2}[N_{13}^5\qt{12}+N_{14}^5\qq{12}];
\notag\\
N_{1t,t}^2&=\dfrac{1}{2}\sq{N_{14}^5(\qt{42}-\qq{23})-
	N_{13}^5(\qt{23}+\qq{42})}; \notag\\
N_{2t,t}^2&=-\dfrac{1}{2}\sq{N_{13}^5(\qt{42}-\qq{23})+
	N_{14}^5(\qt{23}+\qq{42})}; \notag\\
N_{1t,t}^3&=-N_{2t,t}^4=\dfrac{1}{2}(N_{13}^5\qq{34}-N_{14}^5\qt{34});
\notag\\
N_{2t,t}^3&=N_{1t,t}^4=\dfrac{1}{2}(N_{13}^5\qt{34}+N_{14}^5\qq{34});
\notag
\end{align}
\begin{align*}
N_{3t,t}^3&=-\dfrac{1}{2}\sq{N_{13}^5(\qt{13}-\qq{23})+
	N_{14}^5(\qt{23}+\qq{13})};\notag\\
N_{4t,t}^3&=-\dfrac{1}{2}\sq{N_{14}^5(\qt{13}-\qq{23})-
	N_{13}^5(\qt{23}+\qq{13})};\notag\\
N_{3t,t}^4&=\dfrac{1}{2}\sq{N_{14}^5(\qt{42}-\qq{14})-
	N_{13}^5(\qt{14}+\qq{42})};\notag\\
N_{4t,t}^4&=-\dfrac{1}{2}\sq{N_{13}^5(\qt{42}-\qq{14})+
	N_{14}^5(\qt{14}+\qq{42})};\notag\\
N_{1t,t}^5&=-N_{2t,t}^6=-t(\Gamma)_3
-2t(\qt{14}+\qt{23})_4;\notag\\
N_{2t,t}^5&=N_{1t,t}^6=t(\Gamma)_4
-2t(\qt{14}+\qt{23})_3;\notag\\
N_{3t,t}^5&=-N_{4t,t}^6=t(\Gamma)_1
+2t(\qt{14}+\qt{23})_2;\notag\\
N_{4t,t}^5&=N_{3t,t}^6=-t(\Gamma)_2
-2t(\qt{14}+\qt{23})_1;\notag\\
N_{5t,t}^p&=N_{6t,t}^p=0. \notag
\end{align*}
\end{multicols}
\normalsize
The components of $N_{\JJ}$ are

\begin{align} \label{nijcompmin}
\NN^5_{13} &= -t({\qt{13}+\qt{42}+\qq{14}+\qq{23}}) = -2t({A_{22}+A_{33}});\\
\NN^1_{35} &= -\dfrac{2}{t}; \notag\\ \notag\\
\NN^5_{13} &= -\NN^5_{24} = \NN^6_{14} = \NN^6_{23}; \notag\\
\NN^1_{35}&=-\NN^3_{15}=-\NN^4_{16}=\NN^4_{25}=-\NN^3_{26}=
\NN^2_{36}=-\NN^2_{45}=\NN^1_{46}; \notag
\end{align}

For the sake of simplicity, let
\[
\Sigma:=\qt{13}+\qt{42}+\qq{14}+\qq{23}=2(A_{22}+A_{33}).
\]
Therefore, the components of $\nabla N_{\JJ}$ are:
\small
\noindent
\begin{multicols}{2}
\noindent
\begin{align}\label{nablanijmin}
	\NN_{13,a}^5&=-t(\Sigma)_{,a}; \\
	\NN_{13,5}^5&=2(\qq{12}+\qq{34});\notag\\
	\NN_{13,6}^5&=-2(\qt{12}+\qt{34}); \notag\\
	\NN_{13,s}^5&=-\NN_{24,s}^5=\NN_{14,s}^6=\NN_{23,s}^6=0; \notag\\
	\NN_{35,s}^1&=\NN_{15,s}^3=\NN_{16,s}^4=\NN_{25,s}^4=0; \notag\\
	\NN_{26,s}^3&=\NN_{36,s}^2=\NN_{45,s}^2=\NN_{46,s}^1=0; \notag
	\NN_{12,s}^1&=0; \notag\\
	\NN_{13,a}^1&=\dfrac{t^2}{2}\pa{\Sigma-\dfrac{2}{t^2}}\qt{1a}; \notag\\
	\NN_{13,5}^1&=\NN_{13,6}^1=0; \notag\\
	\NN_{14,a}^1&=\dfrac{t^2}{2}\pa{\Sigma-\dfrac{2}{t^2}}\qq{1a}; \notag\\
	\NN_{14,5}^1&=\NN_{14,6}^1=0; \notag\\
	\NN_{15,s}^1&=\NN_{16,s}^1=0; \notag\\
	\NN_{23,a}^1&=\dfrac{t^2}{2}\pa{\Sigma\qq{1a}-
		\dfrac{2}{t^2}\qt{2a}};\notag\\
	\NN_{23,5}^1&=\NN_{23,6}^1=0; \notag\\
	\NN_{24,a}^1&=-\dfrac{t^2}{2}\pa{\Sigma\qt{1a}+
		\dfrac{2}{t^2}\qq{2a}};\notag\\
	\NN_{24,5}^1&=\NN_{24,6}^1=0; \notag\\
	\NN_{25,a}^1&=0; \notag\\
	\NN_{25,5}^1&=\qt{14}+\qt{23}; \notag\\
	\NN_{25,6}^1&=(\qq{14}+\qq{23})-\dfrac{2}{t^2}; \notag\\
	\NN_{26,a}^1&=0; \notag\\
	\NN_{26,5}^1&=\dfrac{2}{t^2}-(\qt{13}+\qt{42}); \notag\\
	\NN_{26,6}^1&=-(\qq{13}+\qq{42})=-\NN_{25,5}^1; \notag\\
	\NN_{34,a}^1&=\qt{4a}-\qq{3a}; \notag\\
	\NN_{34,5}^1&=\NN_{34,6}^1=0; \notag\\
	\NN_{36,a}^1&=0; \notag\\
	\NN_{36,5}^1&=\qt{12}+\qt{34}; \notag\\
	\NN_{36,6}^1&=\qq{12}+\qq{34}; \notag\\
	\NN_{45,s}^1&=-\NN_{36,s}^1; \notag\\
	\NN_{46,s}^1&=0; \notag\\
	\NN_{56,s}^1&=-\NN_{34,s}^1; \notag\\
	\NN_{12,s}^2&=0; \notag\\
	\NN_{13,a}^2&=\dfrac{t^2}{2}\pa{\Sigma\qt{2a}-\dfrac{2}{t^2}\qq{1a}};
	\notag\\
	\NN_{13,5}^2&=\NN_{13,6}^2=0; \notag
		\end{align}
\begin{align*}
	\NN_{14,a}^2&=\dfrac{t^2}{2}\pa{\Sigma\qq{2a}+\dfrac{2}{t^2}\qt{1a}};\\
	\NN_{14,5}^2&=\NN_{14,6}^2=0;\\
	\NN_{15,s}^2&=-\NN_{25,s}^1;\\
	\NN_{16,s}^2&=-\NN_{26,s}^1;\\
	\NN_{23,a}^2&=\dfrac{t^2}{2}\pa{\Sigma-\dfrac{2}{t^2}}\qq{2a};\\
	\NN_{23,5}^2&=\NN_{23,6}^2=0;\\
	\NN_{24,a}^2&=\dfrac{t^2}{2}\pa{\dfrac{2}{t^2}-\Sigma}\qt{2a};\\
	\NN_{24,5}^2&=\NN_{24,6}^2=0;\\
	\NN_{25,s}^2&=\NN_{26,s}^2=0;\\
	\NN_{34,a}^2&=\qt{3a}+\qq{4a};\\
	\NN_{34,5}^2&=\NN_{34,6}^2=0;\\
	\NN_{35,s}^2&=\NN_{46,s}^2=-\NN_{36,s}^1;\\
	\NN_{56,s}^2&=-\NN_{34,s}^2;\\
	\NN_{12,a}^3&=\qq{1a}-\qt{2a};\\
	\NN_{12,5}^3&=\NN_{12,6}^3=0\\
	\NN_{13,a}^3&=\dfrac{t^2}{2}\pa{\Sigma-\dfrac{2}{t^2}}\qt{3a};\\
	\NN_{13,5}^3&=\NN_{13,6}^3=0;\\
	\NN_{14,a}^3&=\dfrac{t^2}{2}\pa{\Sigma\qq{3a}-\dfrac{2}{t^2}\qt{4a}};\\
	\NN_{14,5}^3&=\NN_{14,6}^3=0;\\
	\NN_{16,s}^3&=-\NN_{36,s}^1;\\
	\NN_{23,a}^3&=\dfrac{t^2}{2}\pa{\Sigma-\dfrac{2}{t^2}}\qq{3a};\\
	\NN_{23,5}^3&=\NN_{23,6}^3=0;\\
	\NN_{24,a}^3&=-\dfrac{t^2}{2}\pa{\Sigma\qt{3a}+\dfrac{2}{t^2}\qq{4a}};\\
	\NN_{24,5}^3&=\NN_{24,6}^3=0;\\
	\NN_{25,s}^3&=\NN_{36,s}^1;\\
	\NN_{34,s}^3&=0;\\
	\NN_{35,s}^3&=\NN_{36,s}^3=0;\\
	\NN_{45,s}^3&=\NN_{25,s}^1;\\
	\NN_{46,s}^3&=\NN_{26,s}^1;\\
	\NN_{56,s}^3&=-\NN_{12,s}^3;
\end{align*}
\begin{align*}
	\NN_{12,a}^4&=-(\qt{1a}+\qq{2a});\\
	\NN_{12,5}^4&=\NN_{12,6}^4=0;\\
	\NN_{13,a}^4&=\dfrac{t^2}{2}\pa{\Sigma\qt{4a}-\dfrac{2}{t^2}\qq{3a}};\\
	\NN_{13,5}^4&=\NN_{13,6}^4=0;\\
	\NN_{14,a}^4&=\dfrac{t^2}{2}\pa{\Sigma-\dfrac{2}{t^2}}\qq{4a};\\
	\NN_{14,5}^4&=\NN_{14,6}^4=0;\\
	\NN_{15,s}^4&=\NN_{36,s}^1;\\
	\NN_{23,a}^4&=\dfrac{t^2}{2}\pa{\Sigma\qq{4a}+\dfrac{2}{t^2}\qt{3a}};\\
	\NN_{23,5}^4&=\NN_{23,6}^4=0;\\
	\NN_{24,a}^4&=\dfrac{t^2}{2}\pa{\dfrac{2}{t^2}-\Sigma}\qt{4a};\\
	\NN_{24,5}^4&=\NN_{24,6}^4=0;\\
	\NN_{26,s}^4&=\NN_{36,s}^1;\\
	\NN_{34,s}^4&=0;\\
	\NN_{35,s}^4&=-\NN_{25,s}^1;\\
	\NN_{36,s}^4&=-\NN_{26,s}^1;\\
	\NN_{45,s}^4&=\NN_{46,s}^4=0;\\
	\NN_{56,s}^4&=-\NN_{12,s}^4;\\
	\NN_{12,a}^5&=0;\\
	\NN_{12,5}^5&=\dfrac{t^2}{2}\Sigma(\qt{14}+\qt{23});\\
	\NN_{12,6}^5&=\dfrac{t^2}{2}\Sigma\sq{(\qq{14}+\qq{23})-\dfrac{2}{t^2}};\\
	\NN_{14,a}^5&=0;\\
	\NN_{14,5}^5&=-\dfrac{t^2}{2}\Sigma(\qt{12}+\qt{34});\\
	\NN_{14,6}^5&=-\dfrac{t^2}{2}\Sigma(\qq{12}+\qq{34});
\end{align*}
\begin{align*}
	\NN_{15,s}^5&=-\NN_{13,s}^3;\\
	\NN_{16,s}^5&=-\NN_{14,s}^3;\\
	\NN_{23,s}^5&=\NN_{14,s}^5;\\
	\NN_{25,s}^5&=-\NN_{24,s}^4;\\
	\NN_{26,s}^5&=-\NN_{23,s}^4;\\
	\NN_{34,s}^5&=\NN_{12,s}^5;\\
	\NN_{35,s}^5&=\NN_{13,s}^1;\\
	\NN_{36,s}^5&=\NN_{23,s}^1;\\
	\NN_{45,s}^5&=\NN_{24,s}^2;\\
	\NN_{46,s}^5&=-\NN_{14,s}^2;\\
	\NN_{56,s}^5&=0;\\
	\NN_{12,a}^6&=0;\\
	\NN_{12,5}^6&=\dfrac{t^2}{2}\Sigma\sq{\dfrac{2}{t^2}-(\qt{13}+\qt{42})};\\
	\NN_{12,6}^6&=-\dfrac{t^2}{2}\Sigma(\qq{13}+\qq{42})=-\NN_{12,5}^5;\\
	\NN_{13,s}^6&=-\NN_{14,s}^5;\\
	\NN_{15,s}^6&=-\NN_{13,s}^4;\\
	\NN_{16,s}^6&=-\NN_{14,s}^4;\\
	\NN_{24,s}^6&=\NN_{14,s}^5;\\
	\NN_{25,s}^6&=\NN_{24,s}^3;\\
	\NN_{26,s}^6&=-\NN_{23,s}^3;\\
	\NN_{34,s}^6&=\NN_{12,s}^6;\\
	\NN_{35,s}^6&=\NN_{13,s}^2;\\
	\NN_{36,s}^6&=\NN_{23,s}^2;\\
	\NN_{45,s}^6&=-\NN_{24,s}^1;\\
	\NN_{46,s}^6&=\NN_{14,s}^1;\\
	\NN_{56,s}^6&=0.
\end{align*}
\end{multicols}
\normalsize
\noindent
As far as the two divergences
\[
\diver N_{\JJ}=\NN_{pq,t}^t\theta^p\otimes\theta^q, \qquad
\ol{\diver} N_{\JJ}=\NN_{pt,t}^r\theta^p\otimes e_r.
\]
are concerned, we have, for the first one,
\begin{align} \label{divnijmin1}
\NN_{13,t}^t&=\dfrac{t^2}{2}\pa{\Sigma+\dfrac{2}{t^2}}(\qq{12}+\qq{34})=-\NN_{24,t}^t;\\
\NN_{14,t}^t&=-\dfrac{t^2}{2}\pa{\Sigma+\dfrac{2}{t^2}}(\qt{12}+\qt{34})=\NN_{23,t}^t; \notag
\end{align}
all other components are zero. For the second one, we obtain
\noindent
\begin{multicols}{2}
\noindent
\footnotesize
\begin{align} \label{divnijmin2}
	\NN_{1t,t}^1&=\dfrac{t^2}{2}\pa{\Sigma-\dfrac{2}{t^2}}(\qt{13}+\qq{14});\\
	\NN_{2t,t}^1&=\dfrac{t^2}{2}\pa{\Sigma-\dfrac{2}{t^2}}(\qq{13}-\qt{14});
	\notag\\
	\NN_{3t,t}^1&=\dfrac{t^2}{2}\pa{\dfrac{2}{t^2}-\Sigma}\qq{12}; \notag\\
	\NN_{4t,t}^1&=\dfrac{t^2}{2}\pa{\Sigma-\dfrac{2}{t^2}}\qt{12}; \notag\\
	\NN_{5t,t}^1&=\NN_{6t,t}^1=0; \notag\\
	\NN_{1t,t}^2&=\NN_{2t,t}^1; \notag\\
	\NN_{2t,t}^2&=\dfrac{t^2}{2}\pa{\Sigma-\dfrac{2}{t^2}}(\qt{42}+\qq{23});
	\notag\\
	\NN_{3t,t}^2&=\NN_{4t,t}^1;\notag\\
	\NN_{4t,t}^2&=-\NN_{3t,t}^1;\notag\\
	\NN_{5t,t}^2&=\NN_{6t,t}^2=0;\notag\\
	\NN_{1t,t}^3&=\dfrac{t^2}{2}\pa{\Sigma-\dfrac{2}{t^2}}\qq{34};\notag\\
	\NN_{2t,t}^3&=\dfrac{t^2}{2}\pa{\dfrac{2}{t^2}-\Sigma}\qt{34};\notag\\
	\NN_{3t,t}^3&=\dfrac{t^2}{2}\pa{\Sigma-\dfrac{2}{t^2}}(\qt{13}+\qq{23});
	\notag\\
	\NN_{4t,t}^3&=\dfrac{t^2}{2}\pa{\Sigma-\dfrac{2}{t^2}}(\qq{13}-\qt{23});
	\notag\\
	\NN_{5t,t}^3&=\NN_{6t,t}^3=0;\notag
\end{align}
\begin{align*}
	\NN_{1t,t}^4&=\NN_{2t,t}^3;\\
	\NN_{2t,t}^4&=-\NN_{1t,t}^3;\\
	\NN_{3t,t}^4&=\NN_{4t,t}^3;\\
	\NN_{4t,t}^4&=\dfrac{t^2}{2}\pa{\Sigma-\dfrac{2}{t^2}}(\qt{42}+\qq{14});\\
	\NN_{5t,t}^4&=\NN_{6t,t}^4=0;\\
	\NN_{1t,t}^5&=-t(\Sigma)_{,3};\\
	\NN_{2t,t}^5&=t(\Sigma)_{,4};\\
	\NN_{3t,t}^5&=t(\Sigma)_{,1};\\
	\NN_{4t,t}^5&=-t(\Sigma)_{,2};\\
	\NN_{5t,t}^5&=t^2\pa{\dfrac{2}{t^2}-\Sigma}(\qt{13}+\qt{42});\\
	\NN_{6t,t}^5&=t^2\pa{\dfrac{2}{t^2}-\Sigma}(\qt{14}+\qt{23});\\
	\NN_{1t,t}^6&=-\NN_{2t,t}^5;\\
	\NN_{2t,t}^6&=\NN_{1t,t}^5;\\
	\NN_{3t,t}^6&=-\NN_{4t,t}^5;\\
	\NN_{4t,t}^6&=\NN_{3t,t}^5;\\
	\NN_{5t,t}^6&=\NN_{6t,t}^5;\\
	\NN_{6t,t}^6&=t^2\pa{\dfrac{2}{t^2}-\Sigma}(\qq{14}+\qq{23}).
\end{align*}
\end{multicols}

\section{Acknowledgements}
The first and the third author are members of the Gruppo Nazionale per le
Strutture Algebriche, Geometriche e loro Applicazioni (GNSAGA) of INdAM
(Istituto Nazionale di Alta Matematica).
\bibliographystyle{abbrv}
\bibliography{BiblioTwistor2}

\end{document}